\documentclass[a4paper,10pt]{article}
\usepackage[utf8x]{inputenc}
\pdfoutput=1

\usepackage{nccfoots}
\usepackage{times}
\usepackage{geometry}
\usepackage[T1]{fontenc}
\usepackage{color}
\usepackage{amsmath}
\usepackage{amssymb}
\usepackage{array}
\usepackage{amsthm}
\usepackage{graphicx}
\usepackage{hyperref}
\usepackage{setspace}
\usepackage{stmaryrd}
\usepackage{longtable}
\usepackage{tabularx}
\usepackage{multirow}
\usepackage{caption}

\theoremstyle{plain}
\newtheorem{thm}{Theorem}[section]
\newtheorem{prop}[thm]{Proposition}
\newtheorem{cor}[thm]{Corollary}
\newtheorem{lem}[thm]{Lemma}

\theoremstyle{definition}

\theoremstyle{remark}
\newtheorem{rem}[thm]{Remark}

\theoremstyle{plain}

\newcommand{\R}{\mathbb{R}}

\newcommand{\N}{\mathbb{N}}

\newcommand{\CP}{\mathbb{C}P}
\newcommand{\HP}{\mathbb{H}P}

\newcommand{\dimn}{\mathrm{dim}}

\newcommand{\scal}{\mathrm{scal}}
\newcommand{\ric}{\mathrm{Ric}}
\newcommand{\trace}{\mathrm{tr}}

\newcommand{\volume}{\mathrm{vol}}

\newcommand{\dv}{\text{ }dV}

\newcommand{\spectrum}{\mathrm{spec}}

\renewcommand{\title}[1]{{\bfseries #1}\par}
\renewcommand{\author}[1]{\medskip{#1}\par\smallskip}
\newcommand{\affiliation}[1]{{\itshape #1}\par}
\newcommand{\email}[1]{E-mail:~\texttt{#1}\par}

\numberwithin{equation}{section}

\begin{document}
\begin{center}
\title{\LARGE Stability of $\sin$-cones and $\cosh$-cylinders}
\vspace{3mm}
\author{\Large Klaus Kröncke}
\vspace{3mm}
\affiliation{Universität Hamburg, Fachbereich Mathematik\\Bundesstraße 55\\20146 Hamburg, Germany}

\email{klaus.kroencke@uni-hamburg.de} 
\end{center}
\vspace{2mm}
\begin{abstract}
	This work concerns stability and instability of Einstein warped products with an Einsteinian fiber of codimension $1$. 
	We study the cases where the scalar curvature of the warped product and of the fiber are either both positive or both negative to complement the results in \cite{Kro15c}.
Up to a small gap in the case of $\sin$-cones, the stability properties of such warped products are now completely determined by spectral properties of the Laplacian and the Einstein operator of the fiber.	
For $\cosh$-cylinders, we are furthermore able to prove a convergence result for the Ricci flow starting in a small neighbourhood.
As an interesting class of examples, we determine the stability properties of $\sin$-cones over symmetric spaces of compact type.
\end{abstract}

\section{Introduction}
A Riemannian manifold $(M,g)$ is called Einstein, if the Ricci tensor of the metric satisfies the equation $\ric_g=\lambda\cdot g$ for some constant $\lambda\in\R$. Einstein manifolds are of great interest in differential geometry (see \cite{LW99,Joy00,Bes08} for extensive information) as well as in theoretical physics (see e.g. \cite{GPY82,GHP03}). They are the critical points of the Einstein-Hilbert action $g\to \int_M \scal_g\dv_g$ under volume constraint and stationary points of Hamilton's Ricci flow $\dot{g}(t)=-2\ric_{g(t)}$ on the space of metrics modulo rescalings.

In both contexts, there are corresponding notions of stability which are closely related to each other (see e.g.\ \cite{CH13}). We are working with the notion of (linear) stability which is used in the context of Ricci flow: Let $\hat{g}$ be an Einstein metric with Einstein constant $\lambda$ and the vector field $V=V(g,\hat{g})$ depending on the metrics $g$ and $\hat{g}$ be defined by $V^k=g^{ij}(\Gamma_{ij}^k-\hat{\Gamma}_{ij}^k)$. Then, $\hat{g}$ is a stationary point of the $\lambda$-Ricci-de-Turck flow
\Footnotetext{}{2010 \emph{Mathematics Subject Classification.} 53C25, 58J05, 53C44.}
\Footnotetext{}{\emph{Key words and phrases.} Einstein metrics, stability, $\sin$-cones, $\cosh$-cylinders.}
\begin{align}
\dot{g}(t)=-2\ric_{g(t)}+2\lambda g(t)+L_{V(g(t),\hat{g})}g(t)
\end{align}
and its linearization at $\hat{g}$ is given by
\begin{align}
\frac{d}{dt}\bigg \vert_{t=0}\left[-2\ric_{\hat{g}+th}+2\lambda (g+th)+L_{V(g+th,\hat{g})}(g+th)\right]=-\nabla^*\nabla h+2\mathring{R}h=:-\Delta_Eh,
\end{align}
where $\mathring{R}h_{ij}=R_{iklj}h^{kl}$. We call the elliptic operator $\Delta_E$ defined on the right hand side the Einstein operator. The Einstein operator is closely related to the Lichnerowicz Laplacian, which is given by $\Delta_Lh=\Delta_Eh+\ric\circ h+h\circ\ric$.
Let $S^2M$ be the bundle of symmetric $2$-tensors. We call an Einstein manifold strictly stable if there exists a constant $C>0$ such that
\begin{align}\label{strictlystable}
\int_{M}\langle \Delta_E h,h\rangle \dv\geq C \left\|h\right\|_{L^2}^2\qquad 
\end{align}
for all compactly supported $h\in C^{\infty}(S^2M)$ satisfying $\int_M\trace h\dv=0$ and $\delta h=0$ where $\delta h$ is the divergence of $h$. We call $(M,g)$ stable, if \eqref{strictlystable} holds with $C=0$ and unstable, if it is not stable. Here, the conditions $\int_M\trace h\dv=0$ and $\delta h=0$ refer to volume-preserving perturbations orthogonal to the orbit of the diffeomorphism group acting on $g$.

This (linear) stability problem was initiated by Koiso \cite{Koi78} studied extensively by various authors, see e.g. \cite{Bes08,DWW05,Kro15} and references therein.
To give some examples, we mention that the round sphere, the hyperbolic space and their quotients are strictly stable. The flat euclidean space and $\CP^n$ are stable but not strictly stable. Ricci-flat manifolds with special holonomy are stable. Any product of positive Einstein metrics is unstable. An open problem in this context is the question, whether there exists an unstable compact Einstein metric of nonpositive scalar curvature \cite{Dai07}. In the complete noncompact case, unstable Einstein metrics of nonpositive scalar curvature are known \cite{GPY82,War06}.

To prove (dynamical) stability of Einstein metrics under Ricci flow, linear stability appears to be a nessecary condition.
Such problems have been considered in the compact case e.g.\ in \cite{Ses06,Kro13,HM14,Kro15b} (see also references therein) and in the noncompact case in \cite{SSS08,SSS11,Bam15,Der15}.

%

The study of stability of Einstein warped products was initiated in \cite{HHS14}, where instability of some Ricci-flat cones was proven. (In-)Stability of compact Einstein warped products was also recently considered in \cite{BHM16}. In \cite{Kro15c}, a systematic methology was developed to characterize stability and instability of Einstein warped products by properties of the fiber (if it is Einstein and of codimension $1$). The machinery was applied to Ricci-flat, hyperbolic and exponential cones. The aim of the present paper is to close a gap and to determine the stability of the remaining examples of Einstein warped products with Einsteinian fiber of codimension $1$. The possible warping functions are collected in the table below. Without loss of generality, we have determined the absolute values of the nonvanishing Einstein constants.
	\begin{table}[h]
		\begin{center}
			\renewcommand{\arraystretch}{1.5}
			\begin{tabular}{|l|l|l|l|}
				\hline
				 $\tilde{g}=dr^2+f(r)^2g$ & $\ric_g=(n-1)g$ & $\ric_g=0$ & $\ric_g=-(n-1)g$  \\
				\hline
				$\ric_{\tilde{g}}=n\cdot \tilde{g} $ & $f(r)=\sin(r)$ &  &  \\
			$\ric_{\tilde{g}}=0 $ & $f(r)=r$ & $f(r)=1$ &  \\
				$\ric_{\tilde{g}}=-n\cdot \tilde{g} $ & $f(r)=\sinh(r)$ & $f(r)=e^{r}$ & $f(r)=\cosh(r)$ \\			
               \hline				
							\end{tabular}
							
							\caption{Warping functions for $n+1$-dimensional Einstein warped products}
						\end{center}
						\end{table}

\noindent The cases below the diagonal have been considered in \cite{Kro15b}. In this paper, we consider the cases of the diagonal. The manifolds are called $\sin$-cones if $f(r)=\sin(r)$ and cylinders if $f(r)=1$. In the case $f(r)=\cosh(r)$, we call them $\cosh$-cylinders.
The case $f(r)=1$ is easy and will be discussed in Remark \ref{trivialremark}.
The $\sin$-cones also appears in sting theory \cite{GLNP11,BILPS14} and stability properies of them may be also of great interest in physical contexts.
\begin{thm}\label{thmcoshcone}
Let $(M^n,g)$ be a complete Einstein manifold of scalar curvature $-n(n-1)$. Then the $\cosh$-cylinder
\begin{align}
(\widetilde{M},\tilde{g})=(\R\times M,dr^2+\cosh^2(r)g)
\end{align}
is stable if and only if $\spectrum_{L^2}(\Delta_E|_{TT})\geq -n$ and strictly stable if and only if $\spectrum_{L^2}(\Delta_E|_{TT})>-n$.
\end{thm}
If $(M,g)$ is complete and of bounded curvature then the same holds for $(\tilde{M},\tilde{g})$.
In this case, we are able to prove a stability assertion under the adapted Ricci flow
\begin{align}\label{nrf}
\dot{g}(t)=-2\ric_{g(t)}-2ng(t).
\end{align}
\begin{thm}\label{thmricciflow}
Let $(M,g)$ be a complete Einstein manifold of scalar curvature $-n(n-1)$ such that
\begin{align}
\left\|R_g\right\|_{L^\infty(g)}<\infty,\qquad i(M)>0,\qquad \spectrum_{L^2}(\Delta_E|_{TT})>-n.
\end{align}	
	Then the manifold $(\widetilde{M},\tilde{g})$ from above is stable under the Ricci flow \eqref{nrf} in the following sense: For any $K>0$, there exists an $\epsilon(K,n)>0$ such that the Ricci flow $\tilde{g}(t)$ starting at a metric $\tilde{g}(0)$ satisfying
	\begin{align}
	\left\|\tilde{g}(0)-\tilde{g}\right\|_{L^2(\tilde{g})}\leq K,\qquad 	\left\|\tilde{g}(0)-\tilde{g}\right\|_{L^\infty(\tilde{g})}\leq \epsilon
	\end{align}
	exists for all time and there exists a family of diffeomorphisms $\varphi_t$, $t\geq0$ such that 
		\begin{align}
		\left\|\varphi_t^*\tilde{g}(t)-\tilde{g}\right\|_{C^k(\tilde{g})}\leq C(k)\cdot e^{-\alpha t}
		\end{align}
	for some constants $C(k),\alpha>0$ and all $t\geq 0$.
	\end{thm}
	Since negative K\"{a}hler-Einstein manifolds and Einstein manifolds of nonpositive sectional curvature are stable \cite{Dai07,Koi78}, one gets
	\begin{cor}
		Let $(M^n,g)$ be a negative Einstein manifold of bounded curvature and positive injectivity radius which is either K\"{a}hler or of nonpositive sectional curvature. Then its $\cosh$-cylinder is stable under the Ricci flow in the above sense.
		\end{cor}
\noindent	For $\sin$-cones, we prove the following
\begin{thm}\label{thmsincone}
Let $(M^n,g)$ be a closed Einstein manifold of scalar curvature $n(n-1)$. Then the $\sin$-cone
\begin{align}
(\widetilde{M},\tilde{g})=((0,\pi)\times M,dr^2+\sin^2(r)g)
\end{align}
is (strictly) stable if $(M,g)$ satisfies  $\spectrum(\Delta_E|_{TT})\geq0$ (resp.\ $\spectrum(\Delta_E|_{TT})>0$) and if all nonzero eigenvalues of the Laplacian on $M$ satisfy the bound $\lambda\geq 2n-1$. On the other hand $(\widetilde{M},\tilde{g})$ is unstable if $\spectrum(\Delta_E|_{TT})\ngeq0$ or if there exists a Laplacian eigenvalue on $M$ satisfying the bounds \begin{align}n<\lambda <2n-\frac{n}{2}\left(\sqrt{1+\frac{8}{n}}-1\right).
\end{align}
\end{thm}
\begin{rem}
Note that $-1>-\frac{n}{2}(\sqrt{1+\frac{8}{n}}-1)>-2$ for all $n\in \N$ and $-\frac{n}{2}(\sqrt{1+\frac{8}{n}}-1)\to -2$ as $n\to\infty$. Thus, the theorem leaves a rather unsatisfactory gap where it is not clear whether the $\sin$-cone is stable or not. We are still able to handle most known interesting examples, including all symmetric spaces of compact type (see Section \ref{symmspaces}). We are not sure how optimal the bound $\lambda\geq 2n-1$ in the stability assertion is. However, it is the optimal lower bound of the form $2n-c$, $c\in \R$ that we can prove with our methods. It should also be noted that under the above assumptions, all nonzero eigenvalues on $M$ satisfy $\lambda\geq n$ and equality only holds for the standard sphere which is known to be strictly stable.
\end{rem}
Since $2n-\frac{n}{2}(\sqrt{1+\frac{8}{n}}-1)>2(n-1)$ and any K\"{a}hler-Einstein Fano manifold with a holomorphic vector field admits $2(n-1)$ as a Laplacian eigenvalue, we get
\begin{cor}
	The $\sin$-cone over every K\"{a}hler-Einstein Fano manifold with a holomorphic vector field is unstable. In particular, the $\sin$-cone over $\CP^n$ is unstable for $n>1$.
\end{cor}

In the proof of Theorem \ref{thmcoshcone}, one direction is much easier to show than the other one. If the condition on the Einstein operator of the fiber is not satisfied, one is able to construct a test section $\tilde{h}\in C^{\infty}_{cs}(S^2\widetilde{M})$ violating the stability condition.
It is much harder to prove the converse direction. Here, we decompose the action of $\Delta_E$ on symmetric $2$-tensors into four different components and by tedious calculations, we prove positivity of $\Delta_E$ on each of them. The strategy is the same as in the proofs of the main results in \cite{Kro15c}.

The assertion of Theorem \ref{thmricciflow} is a conseqence of the fact that under the assumptions of the theorem, $(\tilde{M},\tilde{g})$ is a complete strictly stable Einstein manifold of bounded curvature and positive injectivity radius. The proof is very similar to the proof of stability of hyperbolic space under Ricci flow \cite{SSS11}.

The proof of Theorem \ref{thmsincone} is somewhat more involved than the one of Theorem \ref{thmcoshcone}. If the Einstein operator of the fiber restricted to $TT$-tensors has a negative eigenvalue, one can similarly construct a destabilizing perturbation on the cone as above. However, if the eigenvalue condition
\begin{align}n<\lambda <2n-\frac{n}{2}\left(\sqrt{1+\frac{8}{n}}-1\right)
\end{align}
  holds on the fiber, one has to argue differently. We then construct an eigenfunction of the Laplacian on the $\sin$-cone whose eigenvalue is below $2n$. This eigenfunction is shown to be in the closure of $C^{\infty}_{cs}(\tilde{M})$ under the $H^3$-norm. From a sequence of approximating functions, we are then able to construct destabilizing perturbations (we use the divergence-free part of the corresponding conformal perturbations). To prove the stability assertion, we use the same decomposition of symmetric $2$-tensors as above. In this case we need the additional condition on the Laplacian spectrum given in the theorem to ensure that the Einstein operator is positive on all parts of the decomposition.


This paper is organized as follows. In Section \ref{section2}, we recall the decomposition of the space $C^{\infty}_{cs}(S^2\widetilde{M})$ with respect to which the quadratic form $h\mapsto (\Delta_Eh,h)_{L^2}$ has a block diagonal form and we recall how it acts on the blocks.
In the next sections, we prove the theorems \ref{thmcoshcone}, \ref{thmricciflow} and \ref{thmsincone}, respectively. Finally, in Section \ref{symmspaces}, we use Theorem \ref{thmsincone} to determine the stability of $\sin$-cones over symmetric spaces of compact type.

To finish the introduction, we fix some notation and conventions.
The Riemann curvature tensor is defined by the sign convention such that $R_{ijkl}=g(\nabla_{\partial_i}\nabla_{\partial_j}\partial_k-\nabla_{\partial_j}\nabla_{\partial_i}\partial_k,\partial_l)$. The Ricci curvature and the scalar curvature of a metric $g$ are denoted by $\ric_g,\scal_g$, respectively. The rough Laplacian acting on smooth sections of a vector bundle is $\Delta=\nabla^*\nabla=-g^{ij}\nabla^2_{ij}$. The symmetric tensor product is $h\odot k=h\otimes k+k\otimes h$. The divergence of a symmetric $2$-tensor and of a one-form are given by $\delta h_j=-g^{ik}\nabla_ih_{kj}$ and $\delta\omega=-g^{ij}\nabla_i\omega_j$, respectively. The formal adjoint $\delta^*:C^{\infty}(T^*M)\to C^{\infty}(S^2M)$ is $(\delta^*\omega)_{ij}=\frac{1}{2}(\nabla_i\omega_j+\nabla_j\omega_i)$. The space of smooth and compactly supported sections of a vector bundle $E$ is denoted by $C^{\infty}_{cs}(E)$. For notational convenience, we usually denote tensors and differential operators on the warped product manifold by a tilde.
\pagebreak
\section{The Einstein operator on warped products}\label{section2}
In this section, we recall some formulas we developed in \cite[Section 2]{Kro15c}.
Let $(M,g)$ be an Einstein manifold, $I\subset\R$ some open interval and $\tilde{M}=I\times M$ with an Einstein metric of the form $\tilde{g}=dr^2+f(r)^2g$ for some positive function $f:I\to\R$.
Let $W=\left\{\omega\in\Omega^1(M)\mid\delta\omega=0\right\}$ be the space of divergence-free one-forms and $TT=\left\{h\in C^{\infty}(S^2M)\mid \delta h=0,\text{ }\trace h=0 \right\}$ the space of transverse traceless tensors (which are usually called $TT$-tensors).
If $M$ is compact, we can expand any
 $\tilde{h}\in C_{sc}^{\infty}(S^2\widetilde{M})$ with compact support  as
 \begin{equation}\label{tensorexpansion3}
 \begin{aligned}
 \tilde{h}=&
 \sum_{i=1}^{\infty}\varphi_i f^2 h_i+\sum_{i=1}^{\infty}\phi_i v_i \tilde{g}
 +\sum_{i=1}^{\infty}\psi_i^{(1)} f^2 \delta^*\omega_i
  +\sum_{i=1}^{\infty}\psi_i^{(2)} \cdot dr\odot f\omega_i\\
 &+\sum_{i=1}^{\infty}\chi_i^{(1)} f^2(n\nabla^2 v_i+\Delta v_i\cdot g)
 +\sum_{i=1}^{\infty}\chi_i^{(2)} \cdot dr\odot \nabla v_i
 +\sum_{i=1}^{\infty}\chi_i^{(3)}\cdot v_i(f^2g-ndr\otimes dr),
 \end{aligned}
 \end{equation}
 where $\varphi_i,\phi_i,\psi_i^{(j)},\chi_i^{(j)}\in C^{\infty}_{cs}(I)$. Furthermore, $v_i,\omega_i, h_i$ are smooth orthonormal bases of the spaces $L^2(M),L^2(W)$ and $L^2(TT)$ which are eigentensors of the Laplacian on functions, the connection Laplacian on $W$ and the Einstein operator on $TT$, respectively.   Let $0=\lambda_0<\lambda_1\leq\lambda_2\leq\ldots$, $\mu_1\leq\mu_2\leq\ldots$ and $\kappa_1\leq\kappa_2\leq\ldots$ be the corresponding eigenvalues so that $\Delta v_i=\lambda_i v_i$, $\Delta \omega_i=\mu_i\omega_i$ and $\Delta_E h_i=\kappa_i h_i$. 
Let the functions $\varphi_i,\phi_i,\psi_i^{(j)},\chi_i^{(j)}\in C^{\infty}_{cs}(I)$ and
	\begin{equation}\label{tensors}
	\begin{aligned}
	\tilde{h}_{1,i}^{(1)}&=\varphi_i f^2\cdot h_i,\qquad 
	\tilde{h}_{2,i}^{(1)}=\phi_i v_i \tilde{g},\qquad
	\tilde{h}_{3,i}^{(1)}=\psi_i^{(1)} f^2\delta^*\omega_i,\\
	\tilde{h}_{3,i}^{(2)}&=\psi_i^{(2)} \cdot dr\odot f\omega_i,\qquad
	\tilde{h}_{4,i}^{(1)}=\chi_i^{(1)} f^2(n\nabla^2 v_i+\Delta v_i\cdot g),\\
	\tilde{h}_{4,i}^{(2)}&=\chi_i^{(2)} \cdot dr\odot f \nabla v_i,\qquad
	\tilde{h}_{4,i}^{(3)}=\chi_i^{(3)}\cdot v_i(f^2g-ndr\otimes dr).
	\end{aligned}
	\end{equation}
		The $L^2$-norms of these tensors are
		\begin{equation}
		\begin{aligned}
		\left\|\tilde{h}_{1,i}^{(1)}\right\|_{L^2(\tilde{g})}^2&=\int_I\varphi_i^2f^ndr,
		&\left\|\tilde{h}_{2,i}^{(1)}\right\|_{L^2(\tilde{g})}^2&=(n+1)\int_I\phi_i^2f^ndr,\\
		\left\|\tilde{h}_{3,i}^{(1)}\right\|_{L^2(\tilde{g})}^2&=\frac{1}{2}\left(\mu_i-\frac{\scal_g}{n}\right)\int_I(\psi_i^{(1)})^2f^ndr, 
		&\left\|\tilde{h}_{3,i}^{(2)}\right\|_{L^2(\tilde{g})}^2&=2\int_I(\psi_i^{(2)})^2f^ndr,\\
		\left\|\tilde{h}_{4,i}^{(1)}\right\|_{L^2(\tilde{g})}^2&= n\lambda_i[(n-1)\lambda_i-\scal_g]\int_I(\chi_i^{(1)})^2f^ndr,
		&\left\|\tilde{h}_{4,i}^{(2)}\right\|_{L^2(\tilde{g})}^2&=2\lambda_i\int_I(\chi_i^{(2)})^2f^ndr,\\
		\left\|\tilde{h}_{4,i}^{(3)}\right\|_{L^2(\tilde{g})}^2&=(n+1)n\int_I(\chi_i^{(3)})^2f^ndr,
		\end{aligned}
		\end{equation}
		and if $\delta_{il}\cdot\delta_{jm}\cdot\delta_{kn}=0$, \begin{align}(\tilde{h}_{i,j}^{(k)},\tilde{h}_{l,m}^{(n)})_{L^2(\tilde{g})}=0.
		\end{align}
		
	\pagebreak
	\noindent	The Einstein operator acts as

	\begin{equation}\label{formulas}
	\begin{aligned}
	(\tilde{\Delta}_E\tilde{h}_{1,i}^{(1)},\tilde{h}_{1,i}^{(1)})_{L^2(\tilde{g})}&=\int_I(\varphi_i')^2f^n dr
	+\kappa_i\int_I\varphi_i^2f^{n-2}dr,\\
	(\tilde{\Delta}_E\tilde{h}_{2,i}^{(1)},\tilde{h}_{2,i}^{(1)})_{L^2(\tilde{g})}&=
	(n+1)\int_{I}(\phi_i')^2f^{n}dr+
	(n+1)\lambda_i\int_{I}\phi_i^2f^{n-2}dr-2\scal_{\tilde{g}}\int_{I}\phi_i^2f^{n}dr,\\
	(\tilde{\Delta}_E\tilde{h}_{3,i}^{(1)},\tilde{h}_{3,i}^{(1)})_{L^2(\tilde{g})}&=
	\frac{1}{2}\left(\mu_i-\frac{\scal_{g}}{n}\right)\int_I((\psi_i^{(1)})')^2f^n dr\\
	&\qquad+ \frac{1}{2}\left(\mu_i-\frac{\scal_{g}}{n}\right)^2\int_I(\psi_i^{(1)})^2f^{n-2}dr,\\
	(\tilde{\Delta}_E\tilde{h}_{3,i}^{(2)},\tilde{h}_{3,i}^{(2)})_{L^2(\tilde{g})}&=2\mu_i\int_{I}(\psi_i^{(2)})^2f^{n-2}dr
	+(2n+6)\int_{I}(\psi_i^{(2)})^2(f')^2f^{n-2}dr\\&\qquad
	+2\int_{I}((\psi_i^{(2)})')^2f^{n}dr -4\int_{I}(\psi_i^{(2)})^2f''f^{n-1}dr,\\
	(\tilde{\Delta}_E\tilde{h}_{4,i}^{(1)},\tilde{h}_{4,i}^{(1)})_{L^2(\tilde{g})}&= n\lambda_i[(n-1)\lambda_i-\scal_g]\int_I((\chi_i^{(1)})')^2f^n dr \\
	&\qquad+n\lambda_i[(n-1)\lambda_i-\scal_g]\left(\lambda_i-2\frac{\scal_g}{n}\right)\int_I(\chi_i^{(1)})^2f^{n-2}dr,\\
	(\tilde{\Delta}_E\tilde{h}_{4,i}^{(2)},\tilde{h}_{4,i}^{(2)})_{L^2(\tilde{g})}&=(2n+6)\lambda_i\int_{I} (\chi_i^{(2)})^2(f')^2f^{n-2}dr+2\lambda_i\int_{I} ((\chi_i^{(2)})')^2f^ndr\\&\qquad+2\lambda_i\left(\lambda_i-\frac{\scal_{g}}{n}\right)\int_{I} (\chi_i^{(2)})^2f^{n-2}dr
	-4\lambda_i\int_{I}(\chi_i^{(2)})^2 f''f^{n-1}dr,\\
	(\tilde{\Delta}_E\tilde{h}_{4,i}^{(3)},\tilde{h}_{4,i}^{(3)})_{L^2(\tilde{g})}&= 
	n\left((n+1)\lambda_i-2\frac{\scal_{g}}{n}\right)\int_I (\chi_i^{(3)})^2 f^{n-2}dr-4n^2\int_I(\chi_i^{(3)})^2 f'' f^{n-1}dr\\&\qquad
	+(n+1)n\int_I((\chi_i^{(3)})')^2f^ndr+2n^2(n+3)\int_I (\chi_i^{(3)})^2(f')^2f^{n-2}dr
	.
	\end{aligned}
	\end{equation}
	Moreover,
	\begin{align}
	(\tilde{\Delta}_E\tilde{h}_{i,j}^{(k)},\tilde{h}_{l,m}^{(n)})_{L^2(\tilde{g})}=0,
	\end{align}
	if $i\neq l$ or $j\neq m$.
	The other off-diagonal terms are
	\begin{equation}
	\begin{aligned}
	(\tilde{\Delta}_E\tilde{h}_{3,i}^{(1)},\tilde{h}_{3,i}^{(2)})_{L^2(\tilde{g})}&=
	-2\left(\mu_i-\frac{\scal_g}{n}\right)\int_I \psi_i^{(1)}\cdot\psi_i^{(2)} f'f^{n-2}dr,\\
	(\tilde{\Delta}_E\tilde{h}_{4,i}^{(1)},\tilde{h}_{4,i}^{(2)})_{L^2(\tilde{g})}&=
	-4[(n-1)\lambda-\scal_g]\lambda\int_I \chi_i^{(1)}\cdot \chi_i^{(2)} f'f^{n-2}dr,\\
	(\tilde{\Delta}_E\tilde{h}_{4,i}^{(1)},\tilde{h}_{4,i}^{(3)})_{L^2(\tilde{g})}&=0,\\
	(\tilde{\Delta}_E\tilde{h}_{4,i}^{(2)},\tilde{h}_{4,i}^{(3)})_{L^2(\tilde{g})}&=
	4(n+1)\lambda_i\int_I \chi_i^{(2)}\cdot\chi_i^{(3)} f'f^{n-2}dr.
	\end{aligned}
	\end{equation}
In other words, the quadratic from $\tilde{h}\mapsto (\tilde{\Delta}_E\tilde{h},\tilde{h})_{L^2(\tilde{g})}$ is diagonal with respect to the $L^2$-orthogonal decomposition
\begin{align}
C^{\infty}_{cs}(S^2\tilde{M})\subset\bigoplus_{i=1}^{\infty} V_{1,i}\oplus
\bigoplus_{i=0}^{\infty} V_{2,i}\oplus
\bigoplus_{i=1}^{\infty} V_{3,i}\oplus
\bigoplus_{i=0}^{\infty} V_{4,i},
\end{align}
where
\begin{equation}
\begin{aligned}
V_{1,i}&=C^{\infty}_{cs}(I)\cdot f^2h_i,\quad
V_{2,i}=C^{\infty}_{cs}(I)\cdot v_i\cdot \tilde{g},\quad V_{3,i}=C^{\infty}_{cs}(I)\cdot f^2\delta^*\omega_i\oplus C^{\infty}_{cs}(I)\cdot dr\odot f\omega_i,\\
V_{4,i}&=C^{\infty}_{cs}(I)\cdot f^2(n\nabla^2v_i+\Delta v_i\cdot g)\oplus C^{\infty}_{cs}(I)\cdot dr\odot f\nabla v_i 
\oplus C^{\infty}_{cs}(I)\cdot v_i(f^2g-ndr\otimes dr).
\end{aligned}
\end{equation}
Thus to prove the theorems \ref{thmcoshcone} and \ref{thmsincone}, we consider the Einstein operator on each of these subspaces separately.

In the case where $M$ is complete and noncompact, one has to argue a little bit careful, which was slightly neglected in our previous paper \cite{Kro15c}.
One first expands $\tilde{h}$ as
\begin{equation}
\begin{aligned}
\tilde{h}&=f^2 h(r)+v(r,x)\tilde{g}+f^2\delta^*\omega(r)+dr\odot f \bar{\omega}(r)+f^2(n {}^g\nabla^2 \hat{v}(r,x)+\Delta_g \hat{v}(r,x)g)\\
&\qquad+dr\odot {}^g\nabla \bar{v}(r,x)+\tilde{v}(r,x)(f^2g-ndr\otimes dr)
\end{aligned}
\end{equation}
where $h\in TT_g$ and $\omega,\bar{\omega}\in W$ are depending on $r$ and $v,\hat{v},\bar{v},\tilde{v}\in C^{\infty}(\tilde{M})$. Because $\tilde{h}$ has compact support, each of the summands has also compact support. One can now let the Einstein operator act on the summands and compute the corresponding $L^2$-scalar products in terms of scalar products containing $h,\omega,\bar{\omega},v,\hat{v},\bar{v},\tilde{v}$. By suitable integration by parts one can always write them in such a way that there is no differential operator acting on the right slot. Then one can expand the terms of the above sum in a dirichlet eigenbasis for the Laplacian on $M$, the connection Laplacian on $W$ and $\Delta_E$ on $TT$ of a suitable bounded domain $\Omega\subset M$ and one obtains exactly the same formulas as in \eqref{formulas}. The eigenvalues depend on the choice of $\Omega$, but they have lower bounds independent of $\Omega$. These are given by the bottom of the spectra of the appearing operators on the whole manifold. 
 To compute lower bounds for the expressions in \eqref{formulas}, we just use the lower spectral bounds. 
Therefore, the argumentations are the same in the compact and in the noncompact case.

\begin{rem}\label{trivialremark}
The easiest case to consider with these formulas is the pure product metric $\tilde{g}=dr^2+g$ either on $\tilde{M}=S^1\times M$ or $\tilde{M}=\R\times M$ in case when $\tilde{g}$ and $g$ are both assumed to be Ricci-flat. In this case $(\tilde{M},\tilde{g})$ is stable if and only if $(M,g)$ is stable. Moreover, $(\tilde{M},\tilde{g})$ is strictly stable if and only if $(M,g)$ is strictly stable and $\tilde{M}=S^1\times M$. Note that Ricci-flat product metrics on $\R\times M$ can never be strictly stable.
\end{rem}
\section{Proof of Theorem \ref{thmcoshcone}}\label{coshconesection}
Because all Einstein manifolds of dimension $n\leq 3$ are of constant curvature, we may assume that $n\geq 4$.
At first, we prove the following
\begin{lem}\label{infima_4}We have
	\begin{equation}
	\begin{aligned}
	\inf_{\varphi\in C^{\infty}_{cs}(\R)}\frac{\int_{\R} (\varphi')^2\cosh^ndr}{\int_{\R} \varphi^2\cosh^{n-2}dr}&=n,\qquad \inf_{\varphi\in C^{\infty}_{cs}(\R)}\frac{\int_{\R} (\varphi')^2\cosh^ndr}{\int_{\R} \varphi^2\cosh^{n}dr}&= n-1,
	\end{aligned}
	\end{equation}
	and the infimuma are  not realized by functions in $ C^{\infty}_{cs}(\R)$ but 
	 by $\varphi(r)=\cosh^{-n}(r)$ and $\varphi(r)=\cosh^{-n+1}(r)$, respectively. Furthermore,
	\begin{equation}
		\begin{aligned}
		\inf_{\varphi\in C^{\infty}_{cs}(\R)}\frac{\int_{\R} (\varphi')^2\cosh^ndr}{\int_{\R} \varphi^2\sinh^2\cosh^{n-2}dr}\geq n-1 .
		\end{aligned}
	\end{equation}	
\end{lem}
\begin{proof}
	We substitute $\psi(r)=\varphi(r)\cosh^{n}(r)$. Then,
	\begin{align}
	\int_{\R} \varphi^2\cosh^{n-2}dr=\int_{\R} \psi^2\cosh^{-n-2}dr,
	\end{align}
	and by integration by parts,
	\begin{equation}
	\begin{aligned}
	\int_{\R} (\varphi')^2\cosh^{n}dr&=\int_{\R} (\psi'\cosh^{-n}-n\psi\sinh\cosh^{-n-1})^2\cosh^{n}dr\\
	&=\int_{\R}(\psi')^2\cosh^{-n}dr+n^2\int_{\R}\psi^2\sinh^2\cosh^{-n-2}dr\\
	&\qquad-n\int_{\R}(\psi^2)'\sinh\cosh^{-n-1}dr\\
	&=\int_{\R}(\psi')^2\cosh^{-n}dr+n^2\int_{\R}\psi^2\sinh^2\cosh^{-n-2}dr\\
	&\qquad +n\int_{\R}\psi^2(\cosh^{-n}-(n+1)\sinh^2\cosh^{-n-2})dr\\
	&=\int_{\R}(\psi')^2\cosh^{-n}dr+n\int_{\R}\psi^2\cosh^{-n-2}dr.
	\end{aligned}
	\end{equation}
	Thus,
\begin{align}
	\inf_{\varphi\in C^{\infty}_{cs}(\R)}\frac{\int_{\R} (\varphi')^2\cosh^ndr}{\int_{\R} \varphi^2\cosh^{n-2}dr}=\inf_{\psi\in C^{\infty}_{cs}(\R)}\frac{\int_{\R} (\psi')^2\cosh^{-n}dr}{\int_{\R} \psi^2\cosh^{-n-2}dr}+n\geq n
\end{align}
and it is immediate that this infimum is realized by $\psi\equiv1$.
To prove the second assertion, we generalize the substitution from above and set $\psi(r)=\varphi(r)\cosh^{p}(r)$ for some $p\in\R$.
	Then,
		\begin{align}
		\int_{\R} \varphi^2\cosh^{n}dr=\int_{\R} \psi^2\cosh^{n-2p}dr.
		\end{align}
		By a similar calculation as above, one gets
		\begin{equation}
		\begin{aligned}
	\int_{\R} (\varphi')^2\cosh^{n}dr&=\int_{\R}(\psi')^2\cosh^{n-2p}dr+p\int_{\R}\psi^2\cosh^{n-2p}dr\\
	&\quad +p(n-p-1)\int_{\R}\psi^2\sinh^2\cosh^{n-2p-2}dr.
		\end{aligned}
		\end{equation}
		The result now follows from setting $p=n-1$. The third statement is an immediate consequence of the second.
\end{proof}
\begin{proof}[Proof of Theorem \ref{thmcoshcone}]

We proceed similarly as in \cite{Kro15c}
and study the Einstein operator as a quadratic form on the subspaces $V_{k,i}$. Let
\begin{align}
\tilde{h}=\varphi \cosh^2(r)h_i\in V_{1,i}.
\end{align}
Then by Lemma \ref{infima_4},
\begin{equation}
\begin{aligned}
(\tilde{\Delta}_E\tilde{h},\tilde{h})_{L^2(\tilde{g})}&=\int_{\R}(\varphi')^2\cosh^n dr
+\kappa_i\int_{\R}\varphi^2\cosh^{n-2}dr\\
&\geq (n+\kappa_i)\int_{\R}\varphi^2\cosh^{n-2}dr\geq0
\end{aligned}
\end{equation}
 if and only if $\kappa_i\geq -n$ for all (Dirichlet) eigenvalues of the Einstein operator.
 
 If all $\kappa_i> -n$, for all $i$, one proves strict stability as follows: Choose $\theta\in (0,1)$ such that $\kappa_i\geq -\theta\cdot n$ for all $i$. Then, by Lemma \ref{infima_4} again,
\begin{equation}
\begin{aligned}
(\tilde{\Delta}_E\tilde{h},\tilde{h})_{L^2(\tilde{g})}&=\int_{\R}(\varphi')^2\cosh^n dr
+\kappa_i\int_{\R}\varphi^2\cosh^{n-2}dr\\
&\geq (1-\theta) \int_{\R}(\varphi')^2\cosh^n dr+(\theta\cdot n+\kappa_i)\int_{\R}\varphi^2\cosh^{n-2}dr\\
&\geq (1-\theta)(n-1)\int_{\R}\varphi^2\cosh^n dr=(1-\theta)(n-1)\left\|\tilde{h}\right\|_{L^2(\tilde{g})}^2.
\end{aligned}
\end{equation} 
\begin{rem}
	Note that if there exists an eigentensor $h\in C^{\infty}(S^2M)$ such that $\Delta_Eh=-nh$, then $\cosh(r)\cdot h\in C^{\infty}(S^2M)$ is an element in the $L^2$-kernel of $\tilde{\Delta}_E$.
	An example for this situation is provided by numerical analysis in \cite{War06}. There, the family of AdS-Taub Bolt$^{-}$ metrics is discussed. It is a family of $4$-dimensional Einstein metrics depending on the parameter $\ell$ with Einstein constant $-3/\ell^2$. Due to rescaling, we look for a solution where $-4/\ell^2$ is an eigenvalue of the Einstein operator. Due to \cite[Figure 1]{War06}, there exists a parameter $\ell_0>0$, for which this is the case and so the Einstein-operator of its $\cosh$-cylinder admits a nontrivial $L^2$-kernel.
\end{rem}
\noindent For
 \begin{align}
 \tilde{h}=\varphi v_i\tilde{g}\in V_{2,i},
 \end{align}
 we have
 \begin{equation}
 \begin{aligned}
 (\tilde{\Delta}_E\tilde{h},\tilde{h})_{L^2(\tilde{g})}&=
 (n+1)\int_{\R}(\varphi')^2\cosh^{n}dr+
 (n+1)\lambda_i\int_{\R}\varphi^2\cosh^{n-2}dr\\&\qquad+2n(n+1)\int_{\R}\varphi^2\cosh^ndr\geq2n\left\|\tilde{h}\right\|_{L^2(\tilde{g})},
 \end{aligned}
 \end{equation}
 so the Einstein operator is always positive on these spaces.
 Next, pick
 \begin{align}
 \tilde{h}= \tilde{h}_1+ \tilde{h}_2=\varphi f^2\delta^*\omega_i+
 \psi \cdot dr\odot f\omega_i \in V_{3,i}.
 \end{align}
 Then we have the scalar products
 \begin{equation}
 \begin{aligned}
 (\tilde{\Delta}_E\tilde{h}_1,\tilde{h}_1)_{L^2(\tilde{g})}&=
 \frac{1}{2}(\mu_i+(n-1))\int_{\R}(\varphi')^2\cosh^{n} dr\\
 &\qquad+ \frac{1}{2}(\mu_i+(n-1))^2\int_{\R}\varphi^2\cosh^{n-2}dr,\\
 (\tilde{\Delta}_E\tilde{h}_2,\tilde{h}_2)_{L^2(\tilde{g})}&=2\mu_i\int_{\R}\psi^2\cosh^{n-2}dr
 +(2n+6)\int_{\R}\psi^2\sinh^2\cosh^{n-2}dr\\
 &\qquad +2\int_{\R}(\psi')^2\cosh^{n}dr-4\int_{\R}\psi^2\cosh^ndr,\\
 &=(2\mu_i-4)\int_{\R}\psi^2\cosh^{n-2}dr+(2n+2)\int_{\R}\psi^2\sinh^2\cosh^{n-2}dr\\
 &\qquad  +2\int_{\R}(\psi')^2\cosh^{n}dr,\\
 (\tilde{\Delta}_E\tilde{h}_1,\tilde{h}_2)_{L^2(\tilde{g})}&=-2(\mu_i+(n-1))\int_{\R}\varphi\psi \sinh\cosh^{n-2}dr,
 \end{aligned}
 \end{equation}
 and the estimates
 \begin{equation}
 \begin{aligned}
 (\tilde{\Delta}_E\tilde{h}_1,\tilde{h}_1)_{L^2(\tilde{g})}&\geq \frac{n-1}{2}(\mu_i+(n-1))\int_{{\R}}\varphi^2 \cosh^ndr\\&\qquad+ \frac{1}{2}
(\mu_i+(n-1))^2\int_{\R}\varphi^2\cosh^{n-2} dr,\\
 (\tilde{\Delta}_E\tilde{h}_2,\tilde{h}_2)_{L^2(\tilde{g})}&\geq 2(n-1)\int_{\R}\psi^2\cosh^ndr+ 2(\mu_i-2)\int_{\R}\psi^2\cosh^{n-2}dr\\&\qquad+(2n+2)\int_{\R}\psi^2\sinh^2\cosh^{n-2}dr,\\
  2| (\tilde{\Delta}_E\tilde{h}_1,\tilde{h}_2)_{L^2(\tilde{g})} |&\leq \frac{1}{2}(\mu_i+(n-1))^2\int_{\R}\varphi^2\cosh^{n-2}dr\\ 
 &\qquad+8\int_{\R}\psi^2\sinh^2\cosh^{n-2}dr.
 \end{aligned}
 \end{equation}
 Because  $\mu_i\geq n-1$ (a Bochner-type argument shows that $\left\|\nabla \omega\right\|_{L^2(g)}^2=2\left\|d\omega\right\|_{L^2(g)}^2+(n-1)\left\|\omega\right\|_{L^2(g)}^2$ holds for any compactly supported divergence-free one-form $\omega$),
 \begin{equation}
 \begin{aligned}
 &(\tilde{\Delta}_E(\tilde{h}_1+\tilde{h}_2),\tilde{h}_1+\tilde{h}_2)_{L^2(\tilde{g})}\geq 
 \frac{n-1}{2}
 (\mu_i+(n-1))\int_{\R}\varphi^2\cosh^{n} dr\\
 & +2(n-1)\int_{\R}\psi^2\cosh^ndr +2(\mu_i-2)\int_{\R}\psi^2\cosh^{n-2}dr\\
 &+(2n-6)\int_{\R}\psi^2\sinh^2\cosh^{n-2}dr\geq  C(n)(\left\|\tilde{h}_1\right\|_{L^2(\tilde{g})}^2+\left\|\tilde{h}_2\right\|_{L^2(\tilde{g})}^2)
  =C(n)\left\|\tilde{h}\right\|_{L^2(\tilde{g})}^2
 \end{aligned}
   \end{equation}
 for some constant $C(n)$ depending only on the dimension.
   Finally, we consider the spaces $V_{4,i}$. Let
   \begin{align}
   \tilde{h}= \tilde{h}_1+ \tilde{h}_2+\tilde{h}_3=\varphi f^2(n\nabla^2 v_i+\Delta v_i\cdot g)+\psi \cdot dr\odot \nabla v_i+\chi\cdot v_i(f^2g-ndr\otimes dr)\in V_{4,i}.
   \end{align}
   We have the scalar products 
   \begin{equation}
   \begin{aligned}
   (\tilde{\Delta}_E\tilde{h}_1,\tilde{h}_1)_{L^2(\tilde{g})}&= (n-1)n\lambda_i(\lambda_i+n)\int_{\R}(\varphi')^2\cosh^n dr \\
   &\qquad+n(n-1)\lambda_i(\lambda_i+n)(\lambda_i+2(n-1))\int_{\R}\varphi^2\cosh^{n-2}dr,\\
   (\tilde{\Delta}_E\tilde{h}_2,\tilde{h}_2)_{L^2(\tilde{g})}&=(2n+6)\lambda_i\int_{\R} \psi^2\sinh^2\cosh^{n-2}dr+2\lambda_i\int_{\R} (\psi')^2\cosh^ndr\\&\qquad+2\lambda_i\left(\lambda_i+(n-1)\right)\int_{\R} \psi^2\cosh^{n-2}dr
   -4\lambda_i\int_{\R}\psi^2\cosh^ndr,
   \\
   (\tilde{\Delta}_E\tilde{h}_3,\tilde{h}_3)_{L^2(\tilde{g})}&= 
   n((n+1)\lambda_i+2(n-1))\int_{\R} \chi^2 \cosh^{n-2}dr \\
   &\qquad
   +(n+1)n\int_{\R}(\chi')^2\cosh^ndr
   +2n^2(n+3)\int_{\R} \chi^2\sinh^2\cosh^{n-2}dr\\
   &\qquad-4n^2\int_{\R}\varphi^2\cosh^ndr,
   \end{aligned}
   \end{equation}
   and
   \begin{equation}
   \begin{aligned}
   (\tilde{\Delta}_E\tilde{h}_1,\tilde{h}_2)_{L^2(\tilde{g})}&= 
   -4(n-1)\lambda_i(\lambda_i+n)\int_{\R}\varphi\psi\sinh \cosh^{n-2}dr,\\
   (\tilde{\Delta}_E\tilde{h}_2,\tilde{h}_3)_{L^2(\tilde{g})}&= 
   4(n+1)\lambda_i\int_{\R}\psi\chi \sinh \cosh^{n-2}dr.
   \end{aligned}
   \end{equation} 
   By Lemma \ref{infima_4}, we have lower estimates
   \begin{equation}
   \begin{aligned}    (\tilde{\Delta}_E\tilde{h}_1,\tilde{h}_1)_{L^2(\tilde{g})}&\geq
   (n-1)^2n\lambda(\lambda+n)\int_{\R} \varphi^2\sinh^2\cosh^{n-2}dr \\ &\qquad + (n-1)n\lambda_i(\lambda_i+n)(\lambda_i+2(n-1))\int_{\R}\varphi^2\cosh^{n-2} dr, \\
      (\tilde{\Delta}_E\tilde{h}_2,\tilde{h}_2)_{L^2(\tilde{g})}&\geq 4n\lambda_i\int_{\R} \psi^2\sinh^2\cosh^{n-2}dr\\
   &\qquad+2\lambda_i(\lambda_i+2n-3)\int_{\R}\psi^2\cosh^{n-2}dr,\\
   (\tilde{\Delta}_E\tilde{h}_3,\tilde{h}_3)_{L^2(\tilde{g})}&\geq
   (n(n+1)\lambda_i+2n(n-1)+n^2(n-3))\int_{\R} \chi^2 \cosh^{n-2}dr\\
   &\qquad+2n^2(n+1)\int_{\R}\chi^2\sinh^2\cosh^{n-2}dr,
   \end{aligned}
   \end{equation}  
   and for the off-diagonal terms, we use the Young inequality to show
   \begin{equation}
   \begin{aligned}
   2|(\tilde{\Delta}_E\tilde{h}_1,\tilde{h}_2)_{L^2(\tilde{g})}| &\leq n(n-1)\lambda_i(\lambda_i+n)^2\int_{\R}\varphi^2\cosh^{n-2}dr\\
   &\qquad  + \alpha^2\frac{n-1}{n}\lambda_i\int_{\R}\psi^2\sinh^2\cosh^{n-2}dr\\
   &\qquad+n(n-1)^2\lambda_i(\lambda_i+n)\int_{\R}\varphi^2\sinh^2\cosh^{n-2}dr\\
   &\qquad+ (4-\alpha)^2\frac{1}{n}\lambda_i(\lambda_i+n)\int_{\R}\psi^2\cosh^{n-2}dr,\\
   2|(\tilde{\Delta}_E\tilde{h}_2,\tilde{h}_3)_{L^2(\tilde{g})}|&\leq 
   \beta^2\frac{n+1}{n}\lambda_i \int_{\R}\psi^2\sinh^2\cosh^{n-2}dr\\
   & \qquad+ n(n+1)\lambda_i\int_{\R}\chi^2\cosh^{n-2}dr\\
   & \qquad  +(4-\beta)^2\frac{n+1}{2n^2}\lambda_i^2\int_{\R}\psi^2\cosh^{n-2}dr\\
   & \qquad +2n^2(n+1)\int_{\R} \chi^2\sinh^2\cosh^{n-2}dr,
   \end{aligned}
   \end{equation}
   where $\alpha,\beta\in (0,4)$ are some constants which we will specify below. Now we get
   \begin{equation}
   \begin{aligned}
   (\tilde{\Delta}_E&(\tilde{h}_1+\tilde{h}_2+\tilde{h}_3),\tilde{h}_1+\tilde{h}_2+\tilde{h}_3)_{L^2(\tilde{g})}\\ &\geq
     \left[4n-\alpha^2\frac{n-1}{n}-\beta^2\frac{n+1}{n}\right]\lambda_i\int_{\R} \psi^2\sinh^2\cosh^{n-2}dr\\
   &+\left[2\lambda_i(\lambda_i+n-3)-(4-\alpha)^2\frac{\lambda_i(\lambda_i+n)}{n}-(4-\beta)^2\frac{n+1}{2n^2}\lambda_i^2\right]\int_{\R} \psi^2\cosh^{n-2}dr.
   \end{aligned}
   \end{equation}
   Elementary calculations show that the right hand side of this inequality is nonnegative if $n\geq 5$ and $\alpha=\beta=2$ and if $n=4$ and $\alpha=4-\sqrt{2}$, $\beta=2$. An analogous argument as done for the spaces $V_{1,i}$ shows that the Einstein operator is strictly stable on these subspaces, i.e.\
      \begin{equation}
       (\tilde{\Delta}_E(\tilde{h}_1+\tilde{h}_2+\tilde{h}_3),\tilde{h}_1+\tilde{h}_2+\tilde{h}_3)_{L^2(\tilde{g})}\geq C(n)\left\|\tilde{h}_1+\tilde{h}_2+\tilde{h}_3\right\|^2_{L^2(\tilde{g})},
       \end{equation}
       and so we are done with the proof of the theorem.
       \end{proof}
	\section{Stability under Ricci-flow}
	This section is devoted to the proof of Theorem \ref{thmricciflow}. The Einstein operator $\Delta_E$ and all covariant derivatives, norms and scalar products are taken with respect to the background metric $\tilde{g}$. To avoid cumbersome notation, we drop the tilde in the notation of tensors and differential operators on $\tilde{M}$ in this section.
	
	\begin{proof}[Proof of Theorem \ref{thmricciflow}]
		The proof follows along the lines of \cite{SSS11} (see also \cite{Der15}) and we sketch it here due to completeness.
   In order to prove convergence of the Ricci flow modulo diffeomorphism it suffices to prove convergence of the $\lambda$-Ricci-de-Turk flow
	\begin{align}\label{rdtflow}
\partial_tg(t)=-2\ric(g(t))+2n g(t)+L_{V(g(t),\tilde{g})}(g(t)) \quad\mbox{on}\quad M\times (0,\infty),
	\end{align}
	where $V(g(t),\tilde{g})=g^{ij}(\Gamma_{ij}^k-\tilde{\Gamma}_{ij}^k)$.
	It can be also written as
	\[ (\partial_t+\Delta_E)h=R_0[h]+\nabla R_1[h],
	\]
	where $g(t)=\tilde{g}+h(t)$ and the nonlinear terms in $h$ are schematically given by
	\begin{equation}
	\begin{aligned}
	R_0[h]&=g^{-1}*h*h*R+g^{-1}*g^{-1}*\nabla h*\nabla h,\\
	R_1[h]&=(g^{-1}-\tilde{g}^{-1})\nabla h,
	\end{aligned}
	\end{equation}
	see e.g.\ \cite[Section 2]{DL16}. For $\epsilon>0$ and $T>0$ there exists a $\delta=\delta(\epsilon,T)>0$
	such that for every $g(0)$ with $\left\|g(0)-\tilde{g}\right\|_{L^{\infty}}<\delta$,
	there exists a unique solution $g(t)$, $t\in [0,T]$ of \eqref{rdtflow} starting in $g(0)$ and satisfying $\left\|g(t)-\tilde{g}\right\|_{L^{\infty}}<\epsilon$ for all $t\in [0,T]$. This fact can be proven exactly as in \cite[Theorem 2.4]{SSS11} and uniqueness holds due to standard arguments, c.f.\ \cite[Section 2.3]{Bam15} and references therein.
	
	
	Because $\spectrum_{L^2}(\Delta_E)\geq C>0$ is positive, one can split $\Delta_E$ as $\Delta_E=\alpha_0\Delta+\bar{\Delta}_E$ such that $\spectrum_{L^2}(\bar{\Delta}_E)\geq \bar{C}>0$ if $\alpha_0>0$ is small enough. By standard estimates and as long as we have the estimate $\left\|h(t)\right\|_{L^{\infty}}\leq \epsilon_{0}=\epsilon_{0}(K,\tilde{g})$, one gets
	\begin{equation}
	\begin{aligned}
	\partial_t \left\| h\right\|_{L^2}^2&=-2(\Delta_Eh,h)_{L^2}+2(R_0[h]+\nabla^{g_0}R_1[h],h)_{L^2}\\
	&=-2\alpha_0\left\|\nabla h\right\|^2_{L^2}+(\bar{\Delta}_Eh,h)_{L^2}+C\left\|h\right\|_{L^{\infty}}\left\|h\right\|^2_{L^{2}}+C\left\|h\right\|_{L^{\infty}}\left\|\nabla h\right\|^2_{L^{2}} \\& \leq -2\alpha \left\| h\right\|^2_{L^2}
	\end{aligned}
	\end{equation}
	which yields
	\begin{align}
	\left\|h(t)\right\|_{L^2}\leq  \left\|h(0)\right\|_{L^2}\cdot e^{-\alpha t} \leq K\cdot e^{-\alpha t}.
	\end{align}
	The solution of \eqref{rdtflow} can be constructed with the help Dirichlet exhaustions. The above $L^2$ a-priori estimate is first developed for Dirichlet Ricci-de-Turck flows and carries over to the limit flow. This works exactly as in \cite[Theorem 3.1 and Corollary 3.2]{SSS11} and \cite[Proposition 4.1 and Corollary 4.2]{Der15}.
	
	 In addition, it follows from higher derivative estimates (see e.g.\ \cite[Corollary 2.3]{Bam15}) that we have $\left\|\nabla h(t)\right\|_{L^{\infty}}\leq C(\tau)\cdot\epsilon$ for all $t\geq\tau$ as long as  $\left\|h(t)\right\|_{L^{\infty}}\leq \epsilon$. Now pick $p(t)\in M$ such that $|h(t)|(p(t))\geq\frac{1}{2}\left\|h(t)\right\|_{L^{\infty}}$. It is elementary that $|h(t)|(q)\geq\frac{1}{4}\left\|h(t)\right\|_{L^{\infty}}$ for all $q\in M$ such that $d(p(t),q)\leq \frac{1}{4} \left\|h(t)\right\|_{L^{\infty}}\cdot  \left\|\nabla h(t)\right\|_{L^{\infty}}^{-1}$. Consequently,
	\begin{align}
	\left\|h(t)\right\|_{L^2}^2\geq \frac{1}{16}\left\|h(t)\right\|_{L^{\infty}}^2\cdot \volume(B(p(t),\frac{1}{4} \left\|h(t)\right\|_{L^{\infty}}\cdot  \left\|\nabla h(t)\right\|_{L^{\infty}}^{-1}))\geq (C\cdot \epsilon)^{-1}\left\|h(t)\right\|_{L^{\infty}}^2
	\end{align}
	and the constant on the right-hand side is independent of $t$ by the upper bound on $\left\|\nabla h(t)\right\|_{L^{\infty}}$ and the lower bound on the injectivity radius. Therefore,
		\begin{align}
		\left\|h(t)\right\|_{L^{\infty}}\leq C\cdot K\cdot \epsilon \cdot e^{-\alpha t} 
		\end{align}
		as long as $\left\|h(t)\right\|_{L^{\infty}}\leq\epsilon_0$.
			If we choose $\epsilon=\epsilon_0/2$ and $T>0$ so large that $C\cdot K \cdot e^{-\alpha T}\leq 1/2$, the flow satisfies $\left\|h(t)\right\|_{L^{\infty}}\leq\epsilon_0/2$ for $t\geq0$. In particular, it exists for all time and satisfies
			\begin{align}
			\left\|g(t)-\tilde{g}\right\|_{L^{\infty}}\leq C \cdot e^{-\alpha t} 
			\end{align}
	By higher derivative estimates one gets
		\begin{align}
		\left\|g(t)-\tilde{g}\right\|_{C^k}\leq C(k) \cdot e^{-\alpha t} 
		\end{align}	
		which finishes the proof of the theorem.
		\end{proof}
		\begin{rem}
			The above proof works for any strictly stable negative Einstein metric with bounded curvature and positive injectivity radius. Therefore, we can also conclude that the  asymptotically hyperbolic manifolds appearing in \cite{Lee06} are stable under the Ricci flow.
			\end{rem}
	
	\section{Proof of Theorem \ref{thmsincone}}
	In this section $(M^n,g)$ always denotes a compact Einstein manifolds of dimension $n\geq 4$ and Einstein constant $(n-1)$. Moreover, $(\tilde{M},\tilde{g})$ denotes its $\sin$-cone, i.e.\
	\begin{align}
	(\tilde{M},\tilde{g})=((0,\pi)\times M,dr^2+\sin^2(r)g).
	\end{align}
	The proof of Theorem \ref{thmsincone} is more complicated than the one of Theorem \ref{thmcoshcone}. This is basically due to positive Ricci curvature of $\sin$-cones which implies that conformal destabilizing perturbations of the metric are possible. On the other hand, the spaces $V_{4,i}$ are also harder to understand since the Einstein operator is not in general positive on these spaces (which contrasts the previous case). For these reasons we have to extend the strategy. If the eigenvalue bound of the theorem is not satisfied, we construct a destabilizing perturbation by an approximation argument. To prove (strict) stability under the assumptions of the theorem, we basically follow the same strategy as before.
	
	\begin{lem}\label{infima_5}We have
		\begin{equation}
			\begin{aligned}
				\inf_{\varphi\in C^{\infty}_{cs}((0,\pi))}\frac{\int_{0}^{\pi} (\varphi')^2\sin^ndr}{\int_{0}^{\pi} \varphi^2\sin^{n-2}dr}=\inf_{\varphi\in C^{\infty}_{cs}((0,\pi))}\frac{\int_{0}^{\pi} (\varphi')^2\sin^ndr}{\int_{0}^{\pi} \varphi^2\sin^{n}dr}=0.
			\end{aligned}
		\end{equation}
			\end{lem}
			\begin{proof}
			Pick for each $\epsilon>0$ a cutoff function $\varphi_{\epsilon}\in C^{\infty}_{cs}(0,\pi)$  satisfying
			\begin{align}
			\varphi_{\epsilon}\equiv 0 \text{ on }(0,\epsilon)\cup (\pi-\epsilon,\pi),\qquad \varphi_{\epsilon}\equiv 1 \text{ on }(2\epsilon,\pi-2\epsilon),\qquad
			|\varphi_{\epsilon}'|\leq \frac{2}{\epsilon}.
			\end{align}	
				Then we have
				\begin{equation}
					\begin{aligned}
				\int_{0}^{\pi} (\varphi_{\epsilon}')^2\sin^ndr&=\int_{\epsilon}^{2\epsilon} (\varphi_{\epsilon}')^2\sin^ndr+\int_{\pi-2\epsilon}^{\pi-\epsilon} (\varphi_{\epsilon}')^2\sin^ndr\\
				&\leq 8\cdot\epsilon^{-2}\int_{\epsilon}^{2\epsilon}r^ndr\\
				&\leq \frac{8}{n+1}\cdot \epsilon^{n-1},
					\end{aligned}
				\end{equation}
				in particular the left hand side converges to zero as $\epsilon\to0$. On the other hand, if $p>0$,
				\begin{align}
			   \int_{0}^{\pi} \varphi_{\epsilon}^2\sin^pdr\to \int_{0}^{\pi} \sin^pdr=C(p)>0
				\end{align}
				as $\epsilon\to 0$ and this finishes the proof of the lemma.
			\end{proof}
		\noindent	We continue with a generalized version of the above estimate.
				\begin{lem}\label{infima_6}For any $\lambda>0$, we have
					\begin{equation}
						\begin{aligned}
							\inf_{\varphi\in C^{\infty}_{cs}((0,\pi))}\frac{\int_{0}^{\pi} (\varphi')^2\sin^ndr+\lambda\int_{0}^{\pi} \varphi^2\sin^{n-2}dr}{\int_{0}^{\pi} \varphi^2\sin^{n}dr}=\lambda+\mu,
						\end{aligned}
					\end{equation}
					where $\mu=-\frac{n-1}{2}+\sqrt{\frac{(n-1)^2}{4}+\lambda}$. Moreover, this infimum is realized by the function $\varphi(r)=\sin^{\mu}(r)$.
				\end{lem}
				\begin{proof}
			    We substitute $\tilde{\varphi}=\sin^{-\mu}\varphi$ so that
			    \begin{align}
			    \lambda\int_0^{\pi}\varphi^2\sin^{n-2}dr=\int_0^{\pi}\tilde{\varphi}^2\sin^{n+2(\mu-1)}dr\qquad \int_0^{\pi}\varphi^2\sin^ndr=\int_0^{\pi}\tilde{\varphi}^2\sin^{n+2\mu}dr
			    \end{align}
			    and
			    \begin{align}
			    (\varphi')^2=(\tilde{\varphi}')^2\sin^{2\mu}+\mu^2\tilde{\varphi}^2\cos^2\sin^{2\mu-2}+\mu(\tilde{\varphi}^2)'\cos\sin^{2\mu-1}.
			    \end{align}
			    Integration by parts now shows that
			    \begin{equation}\begin{split}
			    \int_0^{\pi}(\varphi')^2\sin^ndr&=\int_0^{\pi}(\tilde{\varphi}')^2\sin^{n+2\mu}dr-\mu(n+\mu-1)\int_0^{\pi}\tilde{\varphi}^2\cos^2\sin^{n+2(\mu-1)}dr\\
			    &\qquad+\mu\int_0^{\pi}\tilde{\varphi}^2\sin^{n+2\mu}dr.
			    \end{split}
			    \end{equation}
			    By the definition of $\mu$, $\lambda=\mu(n+\mu-1)$ and we obtain
			    \begin{align}
			   \frac{\int_{0}^{\pi} (\varphi')^2\sin^ndr+\lambda\int_{0}^{\pi} \varphi^2\sin^{n-2}dr}{\int_{0}^{\pi} \varphi^2\sin^{n}dr}=
			   			   \frac{\int_{0}^{\pi} (\tilde{\varphi}')^2\sin^{n+2\mu}dr}{\int_{0}^{\pi} \tilde{\varphi}^2\sin^{n+2\mu}dr}+\lambda+\mu.
			    \end{align}
			    As in the proof of the previous lemma, we can construct a sequence of compactly supported smooth functions such that the right hand side converges to $\lambda+\mu$. This finishes the proof of the lemma.
			\end{proof}
			\noindent As a consequence, we can construct Dirichlet eigenfunctions on the cone.
			\begin{lem}\label{eigenvalue}
				Let $v\in C^{\infty}(M)$ be such that $\Delta_gv=\lambda v$ for some $\lambda>0$. Then,
				\begin{align}
					\inf_{\varphi\in C^{\infty}_{cs}((0,\pi))}\frac{\left\|\tilde{\nabla}(\varphi\cdot v)\right\|_{L^2(\tilde{g})}^2}{\left\|\varphi\cdot v\right\|_{L^2(\tilde{g})}^2}=\lambda+\mu
					\end{align}
					where $\mu=-\frac{n-1}{2}+\sqrt{\frac{(n-1)^2}{4}+\lambda}$. Moreover, this infimum is realized by $\varphi(r)=\sin^{\mu}(r)$. We have
					$\Delta_{\tilde{g}}(\sin^{\mu}\cdot v)=(\lambda+\mu)\cdot\sin^{\mu}\cdot v$.
				\end{lem}
			\begin{proof}
				The first two assertions follow directly from Lemma \ref{infima_6}. Writing the Laplacian of the metric $\tilde{g}$ as $\Delta_{\tilde{g}}=-\sin^{-n}\partial_r(\sin^n\partial_r)+\sin^{-2}\Delta_g$, it can be straightforwardly checked that $\Delta_{\tilde{g}}(\sin^{\mu}\cdot v)=(\lambda+\mu)\cdot\sin^{\mu}\cdot v$.
			\end{proof}
			\begin{rem}
				Note that $\mu$ is a natural number if and only if $\lambda=k(k+n-1)$ for some $k\in\N$, i.e.\ $\lambda$ coincides with the $k$'th Laplacian eigenvalue of the unit sphere. This is what one expects from to the construction of eigenfunctions on the sphere \cite[Chapter III, C.III]{BGM71}.
			\end{rem}
			\begin{lem}\label{approximation}
				Let $\lambda,\mu$ and $v\in C^{\infty}(M)$ be as in the previous lemma. Then, there exists a sequence $\varphi_i\in C^{\infty}_{cs}((0,\pi))$, $i\in\N$ such that the sequence $\varphi_i\cdot v\in C^{\infty}_{cs}(\tilde{M})$ converges in the $H^3(\tilde{g})$-norm to
				 the function $\sin^{\mu}\cdot v\in C^{\infty}(\tilde{M})$.
				\end{lem}
				\begin{proof}
					Let $\psi\in C^{\infty}((0,\pi))$ and consider the function $\psi\sin^{\mu}v\in C^{\infty}(\tilde{M})$. We now compute its $H^3$-norm. At first, 
					\begin{align}
					\left\|\psi\sin^{\mu}v\right\|_{L^2({\tilde{g})}}^2=\int_0^{\pi}\psi^2\sin^{n+2\mu}dr\cdot\left\|v\right\|_{L^2{(g)}}^2.
					\end{align}
					Then we compute
					\begin{equation}\begin{split}
					d(\psi\sin^{\mu}v)&=\psi\sin^{\mu}dv+(\psi'\sin^{\mu}v+\mu\psi \cos\sin^{\mu-1})dr,\\ |d(\psi\sin^{\mu}v)|_{\tilde{g}}^2&=\psi^2\sin^{2(\mu-1)}|dv|_{g}^2+(\psi'\sin+\mu\psi\cos)^2\sin^{2(\mu-1)}v^2,
					\end{split}
					\end{equation}
					so that
					\begin{equation}
					\begin{split}
					\left\|\tilde{\nabla}(\psi\sin^{\mu}v)\right\|_{L^2({\tilde{g})}}^2&=
					\int_0^{\pi}\psi^2\sin^{n+2(\mu-1)}dr\cdot\left\|\nabla v\right\|_{L^2{(g)}}^2\\&\qquad+	\int_0^{\pi}(\psi'\sin+\mu\psi\cos)^2\sin^{n+2(\mu-1)}dr\cdot\left\|v\right\|_{L^2{(g)}}^2\\
					&=\int_0^{\pi}\psi^2\sin^{n+2(\mu-1)}dr\cdot\left\|\nabla v\right\|_{L^2(g)}^2+\int_0^{\pi}(\psi')^2\sin^{n+2\mu}dr\cdot\left\|v\right\|_{L^2(g)}^2\\
					&\qquad -(n+\mu-1)\mu \int_0^{\pi}\psi^2\cos^2\sin^{n+2(\mu-1)}dr\cdot\left\|v\right\|_{L^2{(g)}}^2\\
					&\qquad + \mu\int_0^{\pi}\psi^2\sin^{n+2\mu}dr\left\|v\right\|_{L^2{(g)}}^2.
					\end{split}
					\end{equation}
					For the Hessian, we get
					\begin{equation}\begin{split}
					\tilde{\nabla}^2(\psi\sin^{\mu}v)&=\psi\sin^{\mu}\nabla^2v+\mu\psi \cos\sin^{\mu} v \cdot g+(\psi\sin^{\mu})''v\cdot dr\otimes dr\\
					&\qquad +(\mu\psi\sin^{\mu-1}+\psi'\sin^{\mu}-\psi\cos\sin^{\mu-1})dr\odot d v,\\
					|\tilde{\nabla}^2(\psi\sin^{\mu}v)|^2_{\tilde{g}}&=\psi^2|\nabla^2v|^2_g\sin^{2(\mu-2)}+n\mu^2\psi^2\cos^2\sin^{2(\mu-2)}v^2-2\mu\psi^2\cos\sin^{2(\mu-2)}\lambda v^2\\
					& +((\psi\sin^{\mu})'')^2v^2+2(\mu\psi\sin^{\mu-2}+\psi'\sin^{\mu-1}-\psi\cos\sin^{\mu-2})^2|dv|_g^2,
					\end{split}
					\end{equation}
					which yields
					\begin{equation}
					\begin{split}
					\left\|\tilde{\nabla}^2(\psi\sin^{\mu}v)\right\|_{L^2({\tilde{g})}}^2&=
					\int_0^{\pi}\psi^2\sin^{n+2(\mu-2)}dr\cdot\left\|\nabla^2 v\right\|_{L^2{(g)}}^2\\&\qquad
					-2\mu\int_0^{\pi}\psi^2\cos\sin^{n+2(\mu-2)}\lambda\left\|v\right\|_{L^2{(g)}}^2
					\\&\qquad+n\mu^2	\int_0^{\pi}\psi^2\cos^2\sin^{n+2(\mu-2)}dr\cdot\left\|v\right\|_{L^2{(g)}}^2\\
					&\qquad +\int_0^{\pi}((\psi\sin^{\mu})'')^2v^2\sin^ndr\cdot\left\|v\right\|_{L^2{(g)}}^2\\
					&+2\int_0^{\pi}(\mu\psi\sin^{\mu-2}+\psi'\sin^{\mu-1}-\psi\cos\sin^{\mu-2})^2\sin^ndr\cdot\left\|\nabla v\right\|_{L^2{(g)}}^2.
					\end{split}
					\end{equation}
					Finally, the third derivative is
					\begin{equation}
					\begin{split}
					\tilde{\nabla}^3(\psi&\sin^{\mu}v)=
					\psi\sin^{\mu}\nabla^3v+\mu \psi\cos\sin^{\mu}dv\otimes g+(\psi\sin^{\mu})'''v\cdot dr\otimes dr\otimes dr\\&
					+(\mu\psi+\psi'\sin-\psi\cos)\cos\sin^{\mu}\cdot S_{23}(g\otimes dv)\\
					&+(\psi\sin^{\mu})''dv\otimes dr\otimes dr\\
					&-2(\mu\psi\sin^{\mu-2}+\psi'\sin^{\mu-1}-\psi\cos\sin^{\mu-2})\cos \cdot dv\otimes dr\otimes dr\\
					&+(\mu\psi\sin^{\mu-1}+\psi'\sin^{\mu}-\psi\cos\sin^{\mu-1})'(dr\otimes dv\otimes dr+dr\otimes dr\otimes dv)\\
					& -(\mu\psi\sin^{\mu-2}+\psi'\sin^{\mu-1}-\psi\cos\sin^{\mu-2})\cos(dr\otimes dv\otimes dr+dr\otimes dr\otimes dv)\\
					&+[(\psi\sin^{\mu})'-2\psi \sin^{\mu-1}\cos]dr\otimes \nabla^2v +\mu[(\psi\cos\sin^{\mu}-2\psi\cos^2\sin^{\mu-1})]v\cdot dr\otimes  g\\
					&-\psi\cos\sin^{\mu-1}S_{23}(\nabla^2v\otimes dr)-[\mu\cos^2\sin^{\mu-1}-(\psi\sin^{\mu})''\cos\sin] v\cdot S_{23}(g\otimes dr)\\
					&+(\mu\psi\sin^{\mu-1}+\psi'\sin^{\mu}-\psi\cos\sin^{\mu-1})S_{23}(\nabla^2v\otimes dr),
					\end{split}
					\end{equation}
					where $S_{23}$ of a $(0,3)$-tensor is $S_{23}(T)_{ijk}=T_{ijk}+T_{ikj}$. A careful consideration of all terms shows that the $H^3$-norm can be written as
					\begin{align}\label{h3norm}
					\left\|\psi \sin^{\mu}v\right\|_{H^3(\tilde{g})}^2=\sum_{l=0}^3\sum_{k=0}^1\sum_{i+j\leq 3}C_{ijkl}(\mu,\lambda,n)\int_0^{\pi}(\psi^{(i)})^2\sin^{n+2(\mu-j)}\cos^kdr\left\|\nabla^lv\right\|_{L^2(g)}^2
					\end{align}
					and $C_{ijkl}(\mu,\lambda,n)$ are some constants. For $\epsilon>0$, let $\psi_{\epsilon}\in C^{\infty}_{cs}(0,\pi)$ a cutoff function satisfying
					\begin{align}
					\psi_{\epsilon}\equiv 0 \text{ on }(0,\epsilon)\cup (\pi-\epsilon,\pi),\qquad \psi_{\epsilon}\equiv 1 \text{ on }(2\epsilon,\pi-2\epsilon),\qquad
					|\psi_{\epsilon}^{(k)}|\leq \frac{C}{\epsilon^k}\text{ for }k\in\left\{1,2,3\right\}
					\end{align}
					for some universal constant $C>0$.
					For $i>0$, we now have
					\begin{equation}
					\begin{split}
					\int_0^{\pi}(\psi_{\epsilon}^{(i)})^2\sin^{n+2(\mu-j)}\cos^kdr&=\int_{\epsilon}^{2\epsilon}(\psi_{\epsilon}^{(i)})^2\sin^{n+2(\mu-j)}\cos^kdr\\&\qquad+\int_{\pi-2\epsilon}^{\pi-\epsilon}(\psi_{\epsilon}^{(i)})^2\sin^{n+2(\mu-j)}\cos^kdr\\
					&\leq 2C\epsilon^{-2i} \int_{\epsilon}^{2\epsilon}r^{n+2(\mu-j)}dr\\
					&\leq C(n,\mu,j) \epsilon^{n+1+2(\mu-i-j)}\leq C(n,\mu,j)\cdot \epsilon,
					\end{split}
					\end{equation}
				    where the last inequality holds because $i+j\leq 3$, $n\geq 4$ and $\mu\geq 1$ (the latter holds because $\lambda\geq n$ for any positive eigenvalue on a positive Einstein manifold). An analogous argumentation shows that
				    					\begin{equation}
				    					\begin{split}
				    \left|\int_0^{\pi}(\psi_{\epsilon})^2\sin^{n+2(\mu-j)}\cos^kdr- 					
				    		 \int_{0}^{\pi}\sin^{n+2(\mu-j)}\cos^kdr	\right|	\leq C(n,\mu,j)\cdot \epsilon.	
				    										\end{split}
				    										\end{equation}
				    												    Therefore  we get
				    \begin{align}
				    \left\|\psi_{\epsilon} \sin^{\mu}v\right\|_{H^3(\tilde{g})}\to  \left\|\sin^{\mu}v\right\|_{H^3(\tilde{g})}
				    \end{align}
				    as $\epsilon\to 0$. This proves the lemma.
										\end{proof}
					\begin{rem}
					With an accordingly modified sequence of cutoff functions, we would not be able to prove convergence of this familiy of functions in any $H^k$-norm with $k>3$. Note in particular that elliptic regularity breaks down because $\sin^{\mu}v$ is an eigenfunction of the Laplacian which is	not contained in all Sobolev spaces.
						\end{rem}
						\begin{lem}\label{tensorT}
							Let $(M^n,g)$ be an Einstein manifold and $v\in C^{\infty}_{cs}(M)$ with $\int_M v\dv=0$. Then the tensor
							\begin{align}
							T(v):=(\Delta_g -\frac{\scal}{n})v\cdot g+\nabla^2v
							\end{align}
							satisfies $\delta T(v)=0$ and $\int_M\trace v\dv=0.$ Moreover,
							\begin{align}
							(\Delta_E T(v)),T(v))_{L^2}=(((n-1)\Delta-\scal)(\Delta-\frac{\scal}{n})(\Delta-2\frac{\scal}{n})v, v)_{L^2}.
							\end{align}
							\end{lem}
							\begin{proof}Checking the conditions $\delta T(v)=0$ and $\int_M\trace v\dv=0$ is straightforward.								
								By \cite[Lemma 2.4]{Kro15c}, we get
								\begin{equation}
								\begin{aligned}
							\left\|T(v)\right\|_{L^2}^2&=n((\Delta-\frac{\scal}{n})^2v,v)_{L^2}+\left\|\nabla^2v\right\|_{L^2}^2-2((\Delta-\frac{\scal}{n})\Delta v,v)_{L^2}\\
							&=n((\Delta-\frac{\scal}{n})^2v,v)_{L^2}-((\Delta-\frac{\scal}{n})\Delta v,v)_{L^2}\\
							&=((n-1)\Delta-\scal)(\Delta-\frac{\scal}{n})v, v)_{L^2}.
								\end{aligned}
							\end{equation}
							The last statement of the lemma follows now from the fact that
							$\Delta_ET(v)=T((\Delta-2\frac{\scal}{n})v)$, see e.g.\ \cite[p.\ 6]{Kro15c}. 
								\end{proof}
				\begin{thm}\label{unstableeigenvalue}Let $(M^n,g)$, $n\geq4$ be a positive Einstein manifold normalized such that $\ric_g=(n-1)g$. Suppose there exists an eigenvalue $\lambda\in \spectrum(M,g)$ such that $n<\lambda <2n-\frac{n}{2}(\sqrt{1+\frac{8}{n}}-1)$. Then the $\sin$-cone over $(M,g)$ is unstable.
					\end{thm}
					\begin{proof}
						Let us pick $v\in C^{\infty}(M)$ with $\Delta v=\lambda\cdot v$ and let $\tilde{v}=\sin^{\mu}v$ with $\mu=-\frac{n-1}{2}+\sqrt{\frac{(n-1)^2}{4}+\lambda}$.
						From Lemma \ref{eigenvalue}, we know that $\tilde{\Delta}\tilde{v}=(\lambda+\mu)\tilde{v}=:\tilde{\lambda}\tilde{v}$ and because of the assumptions on $\lambda$, we have $\tilde{\lambda}\in (n+1,2n)$. Let now $\tilde{v}_i=\varphi_i\cdot v\in C^{\infty}_{cs}(\tilde{M})$ where $\varphi_i$ is the sequence from Lemma \ref{approximation}. 
						Because the integral of $v$ is vanishing, the integral of $\tilde{v}_i$ vanishes as well and due to Lemma \ref{tensorT}, $T(\tilde{v}_i)\in C^{\infty}_{cs}(S^2\tilde{M})$ satisfies $\int_{\tilde{M}}\trace T(\tilde{v}_i)dV_{\tilde{g}}=0$ and $\delta T(\tilde{v}_i)=0$.					As $\tilde{v}_i\to v$ in $H^3(\tilde{g})$,
						\begin{equation}\begin{split}
							(\tilde{\Delta}_E T(\tilde{v}_i),T(\tilde{v}_i)_{L^2(\tilde{g})}&
							=	n(\tilde{\nabla}(\tilde{\Delta}-2n)\tilde{v}_i,\tilde{\nabla}(\tilde{\Delta}-n-1)\tilde{v}_i)_{L^2(\tilde{g})}\\ &\qquad-n^2((\tilde{\Delta}-2n)\tilde{v}_i,(\tilde{\Delta}-n-1)\tilde{v}_i)_{L^2(\tilde{g})}\\
							&\to
							n(\tilde{\nabla}(\tilde{\Delta}-2n)\tilde{v},\tilde{\nabla}(\tilde{\Delta}-n-1)\tilde{v})_{L^2(\tilde{g})}\\
							&\qquad -n^2((\tilde{\Delta}-2n)\tilde{v},(\tilde{\Delta}-n-1)\tilde{v})_{L^2(\tilde{g})}
							\\&=n(\tilde{\lambda}-n-1)(\tilde{\lambda}-2n)(\tilde{\lambda}-n)\left\|\tilde{v}\right\|_{L^2(\tilde{g})}^2<0,
						\end{split}
						\end{equation}
						where the last equality follows from integration by parts.
						Therefore, the left hand side must be negative for sufficiently large $i\in\N$ which proves the theorem.	
						\end{proof}
			\begin{prop}\label{propsincone}
				The operator $\tilde{\Delta}_E$ is nonegative on the subspaces $V_{j,i}$, $1\leq j\leq 3$, $i\geq 1$ if and only is $\Delta_E$ is nonnegative on $TT$-tensors and if all nonzero eigenvalues of $\Delta_g$ satisfy the bound $\lambda\geq 2n-\frac{n}{2}(\sqrt{1+\frac{8}{n}}-1)$. $\tilde{\Delta}_E$ is strictly positive on these subspaces if and only if $\Delta_E$ is strictly positive on $TT$-tensors and the eigenvalue bound holds with the strict inequality.
				\end{prop}
			\begin{proof}
		Let
			\begin{align}
				\tilde{h}=\varphi \sin^2(r)h_i\in V_{1,i}.
			\end{align}
			Then by Lemma \ref{infima_5},
			\begin{equation}
				\begin{aligned}
					(\tilde{\Delta}_E\tilde{h},\tilde{h})_{L^2(\tilde{g})}&=\int_{0}^{\pi}(\varphi')^2\sin^n dr
					+\kappa_i\int_{0}^{\pi}\varphi^2\sinh^{n-2}dr\\
					&\geq \kappa_i\int_{0}^{\pi}\varphi^2\sin^{n}dr=\kappa_i\left\|\tilde{h}\right\|_{L^2(\tilde{g})}^2\geq0
				\end{aligned}
			\end{equation}
			for all $\varphi\in C^{\infty}_{cs}((0,\pi))$ if and only if $\kappa_i\geq 0$. If all $\kappa_i>0$, we have strict stability on these subspaces.
					 For
					 \begin{align}
					 \tilde{h}=\varphi v_i\tilde{g}\in V_{2,i},
					 \end{align}
					 we have, by Lemma \ref{infima_6}	
												 \begin{equation}
						 \begin{aligned}
						 (\tilde{\Delta}_E\tilde{h},\tilde{h})_{L^2(\tilde{g})}&=
						 (n+1)\int_{0}^{\pi}(\varphi')^2\sin^{n}dr+
						 (n+1)\lambda_i\int_{0}^{\pi}\varphi^2\sin^{n-2}dr\\&\qquad-2n(n+1)\int_{0}^{\pi}\varphi^2\sin^ndr\\&\geq \left(\lambda_i-\frac{n-1}{2}+\sqrt{\frac{(n-1)^2}{4}+\lambda_i}-2n\right)\left\|\tilde{h}\right\|_{L^2(\tilde{g})}^2
						 \end{aligned}
						 \end{equation}
						and we obviously obtain strict stability under the condition
						\begin{align}
						\lambda_i-\frac{n-1}{2}+\sqrt{\frac{(n-1)^2}{4}+\lambda_i}>2n
						\end{align} for all $i>0$.
											Next, pick
						\begin{align}
						\tilde{h}= \tilde{h}_1+ \tilde{h}_2=\varphi f^2\delta^*\omega_i+
						\psi \cdot dr\odot f\omega_i \in V_{3,i}.
						\end{align}
						Then we have the scalar products
						\begin{equation}
						\begin{aligned}
						(\tilde{\Delta}_E\tilde{h}_1,\tilde{h}_1)_{L^2(\tilde{g})}&=
						\frac{1}{2}(\mu_i-(n-1))\int_{0}^{\pi}(\varphi')^2\sin^{n} dr\\
						&\qquad+ \frac{1}{2}(\mu_i-(n-1))^2\int_{0}^{\pi}\varphi^2\sin^{n-2}dr,\\
						(\tilde{\Delta}_E\tilde{h}_2,\tilde{h}_2)_{L^2(\tilde{g})}&=2\mu_i\int_{0}^{\pi}\psi^2\sin^{n-2}dr
						+(2n+6)\int_{0}^{\pi}\psi^2\cos^2\sin^{n-2}dr\\
						&\qquad +2\int_{0}^{\pi}(\psi')^2\sin^{n}dr+4\int_0^{\pi}\psi^2\sin^ndr,\\
						(\tilde{\Delta}_E\tilde{h}_1,\tilde{h}_2)_{L^2(\tilde{g})}&=-2(\mu_i-(n-1))\int_{0}^{\pi}\varphi\psi \cos\sin^{n-2}dr
						\end{aligned}
						\end{equation}
						and the estimates
						\begin{equation}
						\begin{aligned}
						(\tilde{\Delta}_E\tilde{h}_1,\tilde{h}_1)_{L^2(\tilde{g})}&\geq \frac{1}{2}
						(\mu_i-(n-1))^2\int_{0}^{\pi}\varphi^2\sin^{n-2} dr,\\
						(\tilde{\Delta}_E\tilde{h}_2,\tilde{h}_2)_{L^2(\tilde{g})}&\geq (2n+6)\int_{0}^{\pi}\psi^2\cos^2\sin^{n-2}dr+4\int_0^{\pi}\psi^2\sin^ndr,\\
						2| (\tilde{\Delta}_E\tilde{h}_1,\tilde{h}_2)_{L^2(\tilde{g})} |&\leq \frac{1}{3}(\mu_i-(n-1))^2\int_{0}^{\pi}\varphi^2\cosh^{n-2}dr\\ 
						&\qquad+12\int_{0}^{\pi}\psi^2\sinh^2\cosh^{n-2}dr.
						\end{aligned}
						\end{equation}
						Because  $\mu_i\geq n-1$ (a Bochner-type argument shows that $\left\|\nabla \omega\right\|_{L^2(g)}^2=2\left\|\delta^*\omega\right\|_{L^2(g)}^2+(n-1)\left\|\omega\right\|_{L^2(g)}^2$ holds for any compactly supported one-form $\omega$) and $n\geq4$,
						\begin{equation}
						\begin{aligned}
						(\tilde{\Delta}_E(\tilde{h}_1+\tilde{h}_2),\tilde{h}_1+\tilde{h}_2)_{L^2(\tilde{g})}&\geq 
						\frac{1}{6}
						(\mu_i-(n-1))^2\int_{0}^{\pi}\varphi^2\sin^{n-2} dr+4\int_{0}^{\pi}\psi^2\sin^ndr\\
						&\qquad  +(2n-6)\int_{0}^{\pi}\psi^2\cos^2\sin^{n-2}dr\\
						&\geq 	\frac{1}{6}
						(\mu_i-(n-1))^2\int_{0}^{\pi}\varphi^2\sin^{n} dr+4\int_{0}^{\pi}\psi^2\sin^ndr\\
						&\geq C(n,\min\left\{\mu_i\mid \mu_i>n-1 \right\})\left(\left\|\tilde{h}_1\right\|_{L^2(\tilde{g})}^2+\left\|\tilde{h}_2\right\|_{L^2(\tilde{g})}^2\right)\\
						&=C(n,\min\left\{\mu_i\mid \mu_i>n-1 \right\})\left\|\tilde{h}\right\|_{L^2(\tilde{g})}^2.
						\end{aligned}
						\end{equation}
						Note that for the last inequality, one has to distinguish the cases $\mu_i>n-1$ and $\mu_i=n-1$ but the inequality is true in either case.
						\end{proof}
  It now just remains to consider the spaces $V_{4,i}$. However, it turns out that we can prove nonnegativity of $\tilde{\Delta}_E$ under a lower eigenvalue bound which leaves an unsatisfacory gap in the statement of Theorem \ref{thmsincone}. Let
  \begin{align}\begin{split}
  \tilde{h}= \tilde{h}_1+ \tilde{h}_2+\tilde{h}_3&=\varphi \sin^2(n\nabla^2 v_i+\Delta v_i\cdot g)+\psi \cdot dr\odot \nabla v_i\\ &\quad+\chi\cdot v_i(\sin^2g-ndr\otimes dr)\in V_{4,i}.
\end{split}  
  \end{align}
  We have the scalar products 
  \begin{equation}
  \begin{aligned}
  (\tilde{\Delta}_E\tilde{h}_1,\tilde{h}_1)_{L^2(\tilde{g})}&= (n-1)n\lambda_i(\lambda_i-n)\int_0^{\pi}(\varphi')^2\sin^n dr \\
  &\qquad+n(n-1)\lambda_i(\lambda_i-n)(\lambda_i-2(n-1))\int_0^{\pi}\varphi^2\sin^{n-2}dr,\\
  (\tilde{\Delta}_E\tilde{h}_2,\tilde{h}_2)_{L^2(\tilde{g})}&=(2n+6)\lambda_i\int_{0}^{\pi} \psi^2\cos^2\sin^{n-2}dr+2\lambda_i\int_{0}^{\pi} (\psi')^2\sin^ndr\\&\qquad+2\lambda_i\left(\lambda_i-(n-1)\right)\int_{0}^{\pi} \psi^2\sin^{n-2}dr
  +4\lambda_i\int_0^{\pi}\psi^2\sin^ndr,
  \\
  (\tilde{\Delta}_E\tilde{h}_3,\tilde{h}_3)_{L^2(\tilde{g})}&= 
  n((n+1)\lambda_i-2(n-1))\int_{0}^{\pi} \chi^2 \sin^{n-2}dr \\
  &\qquad
  +(n+1)n\int_{0}^{\pi}(\chi')^2\sin^ndr
  +2n^2(n+3)\int_{0}^{\pi} \chi^2\cos^2\sin^{n-2}dr\\
  &\qquad+4n^2\int_0^{\pi}\varphi^2\sin^ndr
  \end{aligned}
  \end{equation}
  and
  \begin{equation}
  \begin{aligned}
  (\tilde{\Delta}_E\tilde{h}_1,\tilde{h}_2)_{L^2(\tilde{g})}&= 
  -4(n-1)\lambda_i(\lambda_i-n)\int_0^{\pi}\varphi\psi\cos \sin^{n-2}dr,\\
  (\tilde{\Delta}_E\tilde{h}_2,\tilde{h}_3)_{L^2(\tilde{g})}&= 
  4(n+1)\lambda_i\int_0^{\pi}\psi\chi \cos \sin^{n-2}dr.
  \end{aligned}
  \end{equation} 
 These scalar products induces a quadratic form $Q(\lambda_i):(C^{\infty}_{cs}((0,\pi)))^{\oplus ^3}\to \R$ depending on the parameters $\lambda_i$ and $n$. We say that $Q(\lambda_i)$ is strictly positive if \begin{align}Q(\lambda_i)(\varphi,\psi,\chi)\geq C\cdot \left[ \lambda_i(\lambda_i-n)\int_0^{\pi}\varphi^2\sin^ndr+\lambda_i\int_0^{\pi}\psi^2\sin^ndr + \int_0^{\pi}\chi^2\sin^ndr\right]
 \end{align}
  holds on all of $(C^{\infty}_{cs}((0,\pi)))^{\oplus ^3}$. Note that $Q$ is strictly positive for $\lambda_0=0$.
 \begin{prop}\label{positiveQ}
 	If $Q(\lambda_i)$ is (strictly) positive for all $\lambda_i>0$, then $\tilde{\Delta}_E$ is (strictly) stable on the subspaces $V_{4,i}$, $i\geq0$. If $Q$ is not positive semidefinite, $(\tilde{M},\tilde{g})$ is unstable.
 \end{prop}
 \begin{proof}
  The first assertion follows by definition. To prove the second assertion, it suffices to show the following claim: If $\tilde{\Delta}_E$ is (strictly) positive on all $h\in C^{\infty}_{cs}(S^2\tilde{M})$ with  $\int_{\tilde{M}}\trace h\dv_{\tilde{g}}=0$ and $\tilde{\delta}h=0$ then it is also (strictly) positive on all  $h\in C^{\infty}_{cs}(S^2\tilde{M})$ with $\trace h =0$. As a consequence, stability of $(\tilde{M},\tilde{g})$ implies that $Q$ is positive semidefinite.
  To prove the claim, we use the decomposition
  \begin{align}
  \left\{h\in C^{\infty}_{cs}(\tilde{M})\mid \trace h=0\right\}=W_1 \oplus W_2\oplus TT
  \end{align}
  where
  \begin{equation}
  \begin{aligned}
  W_1&=\left\{n\tilde{\nabla}^2v+\tilde{\Delta}v\tilde{g}\mid v\in C^{\infty}_{cs}(\tilde{M}) \right\},\\
  W_2&=\left\{\tilde{\delta}^*\omega\mid \omega\in \Omega^{1}_{cs}(\tilde{M}),\tilde{\delta}\omega=0 \right\},
  \end{aligned}
  \end{equation}
  and $TT$ denotes the space of transverse traceless tensors. This decomposition is $L^2$-orthogonal and is preserved by the Einstein operator. For $S(v)=n\tilde{\nabla}^2v+\tilde{\Delta}v\tilde{g}$, we have
  \begin{align}
  (\tilde{\Delta}_E S(v),S(v))_{L^2(\tilde{g})}=(n+1)n(\tilde{\Delta}(\tilde{\Delta}-n-1)(\tilde{\Delta}-2n)v,v)_{L^2(\tilde{g})}.
  \end{align}
  For a discussion of these facts, see e.g.\ \cite[pp.\ 6--8]{Kro15c}.
  
   We can split $v$ as $v=\varphi_0+w$ where $\varphi_0=\varphi_0(r)$ and $w$ satisfies $\int_M w(r,x)dV_g(x)=0$ for all $r\in (0,\pi)$. Furthemore,
  $w$ can be splitted to $w=\sum_{i\geq 1} \varphi_i\cdot v_i$ where $\varphi_i=\varphi_i(r)$, $v_i\in C^{\infty}(M)$ and $\Delta_gv_i=\lambda_i\cdot v_i$. Here $\lambda_i$, $i\geq 1$ are the nonzero eigenvalues of $\Delta_g$. Note that this splitting is $L^2$-orthogonal and is preserved by the Laplacian.  
   Because $(\tilde{M},\tilde{g})$ is stable, all nonzero eigenvalues of $(M,g)$ satisfy the bound $\lambda_i-\frac{n-1}{2}+\sqrt{\frac{(n-1)^2}{4}+\lambda_i}\geq2n$  due to Theorem \ref{unstableeigenvalue}.
   Because all $\varphi_i v_i $ is compactly supported, we can expand them in a sum of eigenfunctions of the Dirichlet problem on $(\epsilon,\pi-\epsilon)\times M$.
    Due to Lemma \ref{infima_6}, all dirichlet eigenvalues that are used for the expansion of $\varphi_i v_i $ satisfy the bound $\tilde{\lambda}\geq 2n$. Therefore,
    \begin{align}
    (\tilde{\Delta}_E S(\varphi_i v_i),S(\varphi_i v_i))_{L^2(\tilde{g})}\geq0,\qquad \text{for all } i\geq 1.
    \end{align}
    Moreover as $\varphi_0=\varphi_0(r)$, it can be naturally associated to a function $\psi_0$ on the sphere $S^{n+1}$ written as the $\sin$-cone over $S^n$. Due to stability of the sphere,
    	\begin{align}
    	(\tilde{\Delta}_ES(\varphi_0) ,S(\varphi_0))_{L^2(\tilde{g})}=(\bar{\Delta}_ES(\psi_0),S(\psi_0))_{L^2(g_{rd})}\geq 0
    	\end{align}				
    	where $\bar{\Delta}_E$ denotes the Einstein operator of $S^{n+1}$ with the round metric $g_{rd}$.
     Therefore, $\tilde{\Delta}_E$ is nonnegative on $W_1$.
    Moreover, $\tilde{\Delta}_E$ is always nonnegative on $W_2$: For $\tilde{\delta}^*\omega\in W_2$,
  \begin{align}
  \tilde{\Delta}_E\tilde{\delta}^*\omega=\tilde{\delta}^*(\tilde{\nabla}^*\tilde{\nabla}-\frac{\scal_{\tilde{g}}}{n+1})\omega=2\tilde{\delta}^*\tilde{\delta}\tilde{\delta}^*\omega,
  \end{align}
  where the first equality follows e.g.\ from \cite[p.\ 6]{Kro15c} and the second from a calculation. As a consequence, since $\tilde{\delta}^*\omega$ is compactly supported, 
  	\begin{align}
  	(\tilde{\Delta}_E\tilde{\delta}^*\omega ,\tilde{\delta}^*\omega)_{L^2(\tilde{g})}=2\left\|\tilde{\delta}\tilde{\delta}^*\omega\right\|^2_{L^2(\tilde{g})}\geq0
  	\end{align}	
  By assumption, $\tilde{\Delta}_E$ is nonegative on $TT$-tensors and we conclude that it must be nonnegative on all tracefree tensors.
 \end{proof}
 \begin{thm}\label{thmsincone2}
 The $\sin$-cone $	(\widetilde{M},\tilde{g})$
 	is (strictly) stable if and only if $(M,g)$ is (strictly) stable, all nonzero eigenvalues of the Laplacian on $M$ satisfy the bound $\lambda_i >2n-\frac{n}{2}(\sqrt{1+\frac{8}{n}}-1)$ and the quadratic form $Q(\lambda_i)$ is (strictly) positive for all $\lambda_i>0$.
 \end{thm}
 \begin{proof}
 	This follows from Theorem \ref{unstableeigenvalue}, Proposition \ref{propsincone} and Proposition \ref{positiveQ}.
 \end{proof}
   \begin{prop}\label{estimateQ}
  	If $\lambda_i\geq 2n-1$, the quadratic form $Q(\lambda_i)$ is strictly positive.
  \end{prop}
  \begin{proof}
  We define three quadratic forms $Q_i:(C^{\infty}_{cs}((0,\pi)))^{\oplus ^3}\to \R$, $i=1,2,3$ (resp.\ the associated symmetric bilinear forms) componentwise by
  \begin{equation}
  \begin{aligned}
 Q_1((\varphi,0,0),(\varphi,0,0))&=n(n-1)\lambda_i(\lambda_i-n)(\lambda_i-2(n-1))\int_0^{\pi}\varphi^2\sin^{n-2}dr,\\
 Q_1((0,\psi,0),(0,\psi,0))&=[(2n+6)\lambda_i+2\lambda_i\left(\lambda_i-(n-1)\right)-F-G]\int_{0}^{\pi} \psi^2\cos^2\sin^{n-2}dr,\\
 Q_1((\varphi,0,0),(0,\psi,0))&=Q_1((0,\psi,0),(\varphi,0,0))\\
 & =-2(n-1)\lambda_i(\lambda_i-n)\int_0^{\pi}\varphi\psi\cos \sin^{n-2}dr,\\
 Q_2((0,\psi,0),(0,\psi,0))&=F\int_{0}^{\pi} \psi^2\cos^2\sin^{n-2}dr,\\
 Q_2((0,0,\chi),(0,0,\chi))&=n((n+1)\lambda_i-2(n-1))\int_{0}^{\pi} \chi^2 \sin^{n-2}dr,\\
 Q_2((0,0,\chi),(0,\psi,0))&=Q_2((0,\psi,0),(0,0,\chi))=(n+1)\lambda_i\int_0^{\pi}\psi\chi \cos \sin^{n-2}dr,\\
 Q_3((0,\psi,0),(0,\psi,0))&= G\int_{0}^{\pi} \psi^2\sin^{n-2}dr,\\
 Q_3((0,0,\chi),(0,0,\chi))&=2n^2(n+3)\int_{0}^{\pi} \chi^2\cos^2\sin^{n-2}dr,\\
 Q_3((0,0,\chi),(0,\psi,0))&=Q_3((0,\psi,0),(0,0,\chi))=(n+1)\lambda_i\int_0^{\pi}\psi\chi \cos \sin^{n-2}dr,
  \end{aligned}
  \end{equation}
  and the other components are assumed to be zero.  
It is immediate that $Q(\lambda_i)\geq Q_1+Q_2+Q_3$. Now let
\begin{align}
F=\frac{(n+1)^2\lambda_i^2}{n((n+1)\lambda_i-2(n-1))}+\epsilon,\qquad G=\frac{(n+1)^2\lambda_i^2}{2n^2(n+3)}+\epsilon,
\end{align}
where $\epsilon>0$ is some small constant. By substituting $\Psi=\cos\cdot \psi$ and choosing an orthonormal basis of $L^2([0,\pi])$ with respect to the scalar product $(\phi,\psi)=\int_0^{\pi}\phi\psi\sin^{n-2}dr$, one can associate $Q_2$ with the matrix
\begin{align}
\tilde{Q}_2=\begin{pmatrix}F  & (n+1)\lambda_i\\
(n+1)\lambda_i & n((n+1)\lambda_i-2(n-1)) 
\end{pmatrix},
\end{align}
which is positive definite by the choice of $F$ and because $\lambda_i\geq n> 2(n-1)(n+1)^{-1}$ for any $i>0$. Similarly, $Q_3$ is associated with the matrix 
\begin{align}\label{Q3}\tilde{Q}_3=
\begin{pmatrix}G  & (n+1)\lambda_i\\
(n+1)\lambda_i & 2n^2(n+3)
\end{pmatrix},
\end{align} 
which is positive definite and $Q_1$ is associated with the matrix
\begin{align}\tilde{Q}_1=
\begin{pmatrix}n(n-1)\lambda_i(\lambda_i-n)(\lambda_i-2(n-1))  & -2(n-1)\lambda_i(\lambda_i-n)\\
-2(n-1)\lambda_i(\lambda_i-n) & (2n+6)\lambda_i+2\lambda_i\left(\lambda_i-(n-1)\right)-F-G
\end{pmatrix},
\end{align}  
which is positive if $\lambda_i\geq 2n-1$.
\end{proof}
\begin{rem}
 This is the optimal lower bound of the form $2n-c$ (with c a universal constant) we can reach with these methods because for any $c>1$, the determinant of $\tilde{Q}_3$ becomes negative for large $n$ if we insert $\lambda_i=2n-c$. 
 It seems very likely that there is a critical value $\lambda_{crit}(n)\in (2n-2,2n-1)$ with the following property: $Q(\lambda)$ is (strictly) positive for all $\lambda\geq \lambda_{crit}(n)$ (resp.\ $\lambda>\lambda_{crit}(n)$) and not positive for all $\lambda\in (n,\lambda_{crit}(n))$.
\end{rem}
\begin{proof}[Proof of Theorem \ref{thmsincone}]
	This is now a consequence of Theorem \ref{thmsincone2} and Proposition \ref{estimateQ}.
\end{proof}
\section{Symmetric spaces of compact type}\label{symmspaces}
In this section, we study the stability of $\sin$-cones over symmetric spaces of compact type. Based on the results in \cite{CH13}, we are able to determine the stability properties of every such cone.
\begin{thm}
	Let $M=G$ be a simple Lie group.
	 Then the $\sin$-cone over $M$ is strictly stable, if $G$ is one of the following spaces:
		\begin{align}
		\mathrm{Spin}(n)\text{ }(n\geq 6),\qquad \mathrm{E}_6,\qquad\mathrm{E}_7,\qquad\mathrm{E}_8,\qquad \mathrm{F}_4.
		\end{align}
	On the other hand, the $\sin$-cone over $M$ is unstable, if $G$ is one of the following spaces:
			\begin{align}
			\mathrm{SU}(n+1)\text{ }(n\geq 3),\qquad \mathrm{Spin}(5),\qquad\mathrm{Sp}(n)\text{ }(n\geq 3),\qquad\mathrm{G}_2.
			\end{align}
\end{thm}
\begin{proof}The proof is given by the table below based on the results of \cite{CH13}.
		\begin{center}
		\renewcommand{\arraystretch}{1.5}
\begin{longtable}{|l|l|c|c|l|l|}
		\hline
		type & $\mathrm{G}$ & $\dimn(\mathrm{G})$ & $\Lambda$ & stability & cone stability \\
		\hline
		$\mathrm{A}_n$	& $\mathrm{SU}(n+1)$, $n\geq 2$			& $n^2-1$			& $\frac{2n(n+2)}{(n+1)^2}$ & unstable & unstable\\
		\hline
		\multirow{3}{*}{$\mathrm{B}_n$}
		& $\mathrm{Spin}(5)	$		& $10$		& $\frac{5}{3}$ & unstable & unstable\\
		& $\mathrm{Spin}(7)	$		& $21$	& $\frac{21}{10}$ & s.\ stable & s.\ stable \\
		& $\mathrm{Spin}(2n+1)$, $n\geq 4$		& $2n(n+1)$		& $\frac{4n}{2n-1}$ & s.\ stable & s.\ stable\\
		\hline
		$\mathrm{C}_n$	& $\mathrm{Sp}(n)$, $n\geq 3$			& $n(2n+1)$		& $\frac{2n+1}{n+1}$ & unstable & unstable\\
		\hline
		$\mathrm{D}_n$	& $\mathrm{Spin}(2n)$, $n\geq 3$			& $n(2n+1)$		& $\frac{2n-1}{n-1}$ & s.\ stable & s.\ stable\\
		\hline
		$\mathrm{E}_6$	& $\mathrm{E}_6$			& $156$		& $\frac{26}{9}$ & s.\ stable & s.\ stable\\
		\hline
		$\mathrm{E}_7$	& $\mathrm{E}_7$			& $266$		& $\frac{19}{6}$ & s.\ stable & s.\ stable\\
		\hline
		$\mathrm{E}_8$	& $\mathrm{E}_8$			& $496$		& $4$ & s.\ stable & s.\ stable\\
		\hline
		$\mathrm{F}_4$	& $\mathrm{F}_4$			& $52$		& $\frac{8}{3}$ & s.\ stable & s.\ stable\\
		\hline
		$\mathrm{G}_2$	& $\mathrm{G}_2$			& $14$		& $2$ & stable & unstable\\
		\hline	
					\caption{Stability properties of $\sin$-cones over simple Lie groups}
			\end{longtable}
		\end{center}

  Here, $\Lambda:=\lambda_1\cdot (\dimn(G)-1)^{-1}$ is the first nonzero Laplacian eigenvalue normalized by the Einstein constant.
All data except the last column can be found in \cite[Table 1]{CH13}. 
The entries in the last column follow from Theorem \ref{thmsincone}.
\end{proof}
\begin{thm}Let $M=G/K$ be a simply-connected irreducible symmetric space of compact type other than the standard sphere. Then the $\sin$-cone over $M$ is stable, if $G/K=\mathrm{SU}(n)/\mathrm{SO}(n)$, $n\geq 3$. The $\sin$-cone is furthermore strictly stable if $G/K$ is one of the real Grasmannians
	\begin{equation}
	\begin{aligned}
	\frac{\mathrm{SO}(2m+2n+1)}{\mathrm{SO}(2m+1)\times \mathrm{SO}(2n)}\text{ }(n\geq 2,m\geq 1)&,\qquad
	\frac{\mathrm{SO}(8)}{\mathrm{SO}(5)\times\mathrm{SO}(3)},\qquad
	\frac{\mathrm{SO}(2n)}{\mathrm{SO}(n)\times \mathrm{SO}(n)}\text{ }(n\geq 4),\\
	\frac{\mathrm{SO}(2n+2)}{\mathrm{SO}(n+2)\times \mathrm{SO}(n)}\text{ }(n\geq 4)&,\qquad
	\frac{\mathrm{SO}(2n)}{\mathrm{SO}(2n-m)\times \mathrm{SO}(m)}\text{ }(n-2\geq m\geq 3),
	\end{aligned}
	\end{equation}
	or one of the following spaces:
	\begin{equation}
	\begin{aligned}
&\mathrm{E}_6/[\mathrm{Sp}(4)/\left\{\pm I\right\}],\qquad 
\mathrm{E}_6/\mathrm{SU}(2)\cdot \mathrm{SU}(6),\qquad
\mathrm{E}_7/[\mathrm{SU}(8)/\left\{\pm I\right\}],\qquad
\mathrm{E}_7/\mathrm{SO}(12)\cdot\mathrm{SU}(2),\\
&\mathrm{E}_8/\mathrm{SO}(16),\qquad
\mathrm{E}_8/\mathrm{E}_7\cdot \mathrm{SU}(2),\qquad
\mathrm{F}_4/Sp(3)\cdot\mathrm{SU}(2),\qquad
\mathrm{G}_2/\mathrm{SO}(4).
	\end{aligned}
	\end{equation}
	On the other hand, the $\sin$-cone is unstable if $G/K$ is $\CP^n$, $n\geq2$, $\HP^n$, $n\geq2$,
	one of the (real, complex and quaternionic) Grasmannians
	\begin{equation}
	\begin{aligned}
	&\frac{\mathrm{SO}(5)}{\mathrm{SO}(3)\times\mathrm{SO}(2)},\qquad
	\frac{\mathrm{SO}(2n+2)}{\mathrm{SO}(2n)\times \mathrm{SO}(2)}\text{ }(n\geq 3),\qquad
	\frac{\mathrm{SO}(2n+3)}{\mathrm{SO}(2n+1)\times \mathrm{SO}(2)}\text{ }(n\geq 2),\\
	&\frac{\mathrm{U}(m+n)}{\mathrm{U}(m)\times \mathrm{U}(n)}\text{ }(m\geq n\geq 2),\qquad
	\frac{\mathrm{Sp}(m+n)}{\mathrm{Sp}(m)\times \mathrm{Sp}(n)}\text{ }(m\geq n\geq 2)
	\end{aligned}
	\end{equation}
	or one of the following spaces:
		\begin{equation}
		\begin{aligned}
		&\mathrm{SU}(2n)/\mathrm{Sp}(n)\text{ }(n\geq 3),\qquad
		\mathrm{Sp}(n)/\mathrm{U}(n)\text{ }(n\geq 3),\qquad
		\mathrm{SO}(2n)/\mathrm{U}(n)\text{ }(n\geq 5),\\
		&\mathrm{E}_6/\mathrm{SO}(10)\cdot \mathrm{SO}(2),\qquad
		\mathrm{E}_6/\mathrm{F}_4,\qquad
		\mathrm{E}_7/\mathrm{E}_6\cdot \mathrm{SO}(2),\qquad
		\mathrm{F}_4/\mathrm{Spin}(9).
		\end{aligned}
		\end{equation}
	\end{thm}
	\newpage
\begin{proof}This proof is given by the following table:
\begin{center}
	\renewcommand{\arraystretch}{1.5}
	\begin{longtable}{|l|l|c|c|l|l|}
	\hline
	type & $G/K$ & $\dimn(G/K)$ & $\Lambda$ & stability & cone stability \\
	\hline	
	 A I	& $\mathrm{SU}(n)/\mathrm{SO}(n)$, $n\geq 3$			& $\frac{(n-1)(n+2)}{2}$			& $\frac{2(n-1)(n+2)}{n^2}$ &  stable &  stable\\
	\hline
	\multirow{2}{*}{A II}
	& $\mathrm{SU}(4)/\mathrm{Sp}(2)=S^5	$		& $5$		& $\frac{5}{4}$ & s.\ stable & s.\ stable\\
	& $\mathrm{SU}(2n)/\mathrm{Sp}(n)$, $n\geq3	$		& $(n-1)(2n+1)$	& $\frac{(2n+1)(n-1)}{n^2}$ & unstable & unstable \\
	\hline
	\multirow{2}{*}{A III}
	& $\frac{\mathrm{U}(n+1)}{\mathrm{U}(n)\times \mathrm{U}(1)}=\CP^n	$		& $2n$		& $2$ & stable & unstable\\
	& $\frac{\mathrm{U}(m+n)}{\mathrm{U}(m)\times \mathrm{U}(n)}$, $m\geq n\geq2	$		& $2mn$	& $2$ &  stable & unstable \\
	\hline
		\multirow{6}{*}{B I}
			& $\frac{\mathrm{SO}(5)}{\mathrm{SO}(3)\times SO(2)}$		& $6$	& $2$ &  unstable & unstable \\		
		& $\frac{\mathrm{SO}(2n+3)}{\mathrm{SO}(2n+1)\times \mathrm{SO}(2)}$, $n\geq 2$		& $4n+2$	& $2$ &  stable & unstable \\
		& $\frac{\mathrm{SO}(7)}{\mathrm{SO}(4)\times \mathrm{SO}(3)}$		& $12$	& $\frac{12}{5}$ & s.\ stable & s.\ stable \\
		& $\frac{\mathrm{SO}(2n+3)}{\mathrm{SO}(3)\times \mathrm{SO}(2n)}$, $n\geq 3$		& $6n$	& $\frac{4n+6}{2n+1}$ & s.\ stable & s.\ stable \\
		& $\frac{\mathrm{SO}(2m+2n+1)}{\mathrm{SO}(2m+1)\times \mathrm{SO}(2n)}$, $m,n\geq 2$		& $2n(2m+1)$	& $\frac{4m+4n+2}{2m+2n-1}$ & s.\ stable & s.\ stable \\
		\hline
		 B II & $\frac{\mathrm{SO}(2n+1)}{\mathrm{SO}(2n)}=S^{2n}$, $n\geq 1$		& $2n$		& $\frac{2n}{2n-1}$ & s.\ stable & s.\ stable\\	
		\hline
    	C I	& $\mathrm{Sp}(n)/\mathrm{U}(n)$, $n\geq 3$		& $n(n+1)$		& $2$ & unstable & unstable\\		
		\hline
		\multirow{3}{*}{C II}
		& $\frac{\mathrm{Sp}(2)}{\mathrm{Sp}(1)\times \mathrm{Sp}(1)}=S^4	$		& $4$		& $\frac{4}{3}$ & s.\ stable & s.\ stable\\
		& $\frac{\mathrm{Sp}(n+1)}{\mathrm{Sp}(n)\times \mathrm{Sp}(1)}=\HP^n$, $n\geq 2$		& $4n$	& $\frac{2(n+1)}{n+2}$ & unstable & unstable \\
		& $\frac{\mathrm{Sp}(m+n)}{\mathrm{Sp}(m)\times \mathrm{Sp}(n)}$, $m\geq n\geq 2$		& $4mn$	& $\frac{2(m+n)}{m+n+1}$ & unstable & unstable \\
		\hline
	\multirow{6}{*}{D I}
	
	& $\frac{\mathrm{SO}(8)}{\mathrm{SO}(5)\times\mathrm{SO}(3)}$		& $15$	& $\frac{5}{2}$ &  s.\ stable & s.\ stable \\		
	& $\frac{\mathrm{SO}(2n+2)}{\mathrm{SO}(2n)\times \mathrm{SO}(2)}$, $n\geq 3$		& $4n$	& $2$ & stable & unstable \\
	& $\frac{\mathrm{SO}(2n)}{\mathrm{SO}(n)\times \mathrm{SO}(n)}$, $n\geq 4$		& $n^2$	& $\frac{2n}{n-1}$ & s.\ stable & s.\ stable \\
	& $\frac{\mathrm{SO}(2n+2)}{\mathrm{SO}(n+2)\times \mathrm{SO}(n)}$, $n\geq 4$		& $n(n+2)$	& $\frac{2n+2}{n}$ & s.\ stable & s.\ stable \\
	& $\frac{\mathrm{SO}(2n)}{\mathrm{SO}(2n-m)\times \mathrm{SO}(m)}$, $n-2\geq m\geq 3$		& $(2n-m)m$	& $\frac{2n}{n-1}$ & s.\ stable & s.\ stable \\
	\hline
	D II& $\frac{SO(2n+2)}{SO(2n+1)}=S^{2n+1}$, $n\geq 3$		& $2n+1$		& $\frac{2n+1}{2n}$ & s.\ stable & s.\ stable\\
	\hline
	D III	& $\mathrm{SO}(2n)/\mathrm{U}(n)$, $n\geq 5$		& $n(n-1)$		& $2$ &  stable & unstable\\
	\hline
	E I	& $\mathrm{E}_6/[\mathrm{Sp}(4)/\left\{\pm I\right\}]$ & $42$		& $\frac{28}{9}$ & s.\ stable & s.\ stable\\
	\hline
	E II	& $\mathrm{E}_6/\mathrm{SU}(2)\cdot \mathrm{SU}(6)$			& $40$		& $3$ & s.\ stable & s.\ stable\\
	\hline
	E III	& $\mathrm{E}_6/\mathrm{SO}(10)\cdot \mathrm{SO}(2)$	& $32$		& $2$ &  stable & unstable\\
	\hline
	E IV	& $\mathrm{E}_6/\mathrm{F}_4$			& $26$		& $\frac{13}{9}$ & unstable & unstable\\
	\hline
	E V	& $\mathrm{E}_7/[\mathrm{SU}(8)/\left\{\pm I\right\}]$ & $70$		& $\frac{10}{3}$ & s.\ stable & s.\ stable \\
	\hline
	E VI	& $\mathrm{E}_7/\mathrm{SO}(12)\cdot\mathrm{SU}(2)$ & $64$		& $\frac{28}{9}$ & s.\ stable & s.\ stable \\
	\hline
	E VII	& $\mathrm{E}_7/\mathrm{E}_6\cdot \mathrm{SO}(2)$ & $54$		& $2$ & stable & unstable \\
	\hline
	E VIII	& $\mathrm{E}_8/\mathrm{SO}(16)$ & $128$		& $\frac{62}{15}$ & s.\ stable & s.\ stable \\
	\hline
	E IX	& $\mathrm{E}_8/\mathrm{E}_7\cdot \mathrm{SU}(2)$ & $112$		& $\frac{16}{5}$ & s.\ stable & s.\ stable \\
	\hline
	F I	& $\mathrm{F}_4/Sp(3)\cdot\mathrm{SU}(2)$ & $28$		& $\frac{26}{9}$ & s.\ stable & s.\ stable \\
	\hline
	F II	& $\mathrm{F}_4/\mathrm{Spin}(9)$ & $16$		& $\frac{4}{3}$ & unstable & unstable \\
	\hline
	G	& $\mathrm{G}_2/\mathrm{SO}(4)$ & $8$		& $\frac{7}{3}$ & s.\ stable & s.\ stable \\
	\hline
	\caption{Stability properties of $\sin$-cones over symmetric spaces of non-group type}
		\end{longtable}
			
	\end{center}
	Here, all data except the last column can be found in \cite[Table 2]{CH13}.
	The entries in the last column follow from Theorem \ref{thmsincone} except for the case $\mathrm{SU}(3)/\mathrm{SO}(3)$, where the condition $\lambda_i\geq 2\dimn(G/K)-1$ does not hold for the smallest nonzero eigenvalue. However, in this case one can directly check that for all eigenvalues bigger than the given one, the determinant of the matrix $\tilde{Q}_3$ in \eqref{Q3} is positive. Therefore, the quadratic form $Q(\lambda_i)$ is positive for all eigenvalues of $\mathrm{SU}(3)/\mathrm{SO}(3)$ and we can apply Theorem \ref{thmsincone2}.
	\end{proof}

\end{document}